\documentclass[11pt]{article}

\usepackage[scaled=0.92]{helvet}	

\usepackage{amsmath,amssymb}

\usepackage{amsthm}
\usepackage{mathtools}
\usepackage{mathdots}
\usepackage{verbatim}
\usepackage{epigraph}

\usepackage{tikz}
\usetikzlibrary{arrows}

\RequirePackage[colorlinks]{hyperref}

\renewcommand{\P}{\mathbb{P}}  

\newcommand{\ZZ}{\mathbb{Z}}

\newcommand{\given}{ \mid  }
\newcommand{\pr}[1]{\P \event{ #1 }}
\newcommand{\prs}[2]{\P_{#1} \event{ #2 }}

\DeclarePairedDelimiter\abs{\lvert}{\rvert}

\DeclarePairedDelimiter\event{ \{ }{ \} }

\DeclarePairedDelimiter\set{ \{ }{ \} }

\newcommand{\be}{\begin{enumerate}}
\newcommand{\ee}{\end{enumerate}}

\renewcommand{\equiv}{\coloneqq}

\newtheorem{Lem}{Lemma}
\newtheorem{Thm}{Theorem}
\newtheorem{Pro}{Proposition}
\newtheorem{Cor}{Corollary}

\newtheorem{Exp}{Example}

\definecolor{mypurple}{rgb}{.3,0,.5}

\begin{document}
\title{Yaglom limits for $R$-transient chains with non-trivial Martin boundary}
\author{R. D. Foley \and D. R. McDonald}

\maketitle

\begin{abstract}
We give conditions for the existence of a Yaglom limit for $R$-transient Markov chains with non-trivial
$\rho$-Martin entrance boundary
and we characterize the $\rho$-invariant limiting quasistationary distribution ($\rho=1/R$).
\end{abstract}


\footnote{The first author's research was supported in part by NSF Grant CMMI-0856489.  The second author's research was supported in part by NSERC Grant A4551.}

\section{Introduction}

Let $K$ be a  substochastic matrix with elements $K(x,y)$ where $x$
and $y$ are elements of a countably infinite state space $S$. We assume that there is at least one $x \in S$ with $K(x,S) < 1$.
We think of $K$ as the part of the transition matrix of a Markov chain of $X = \set{X_0, X_1,
	\dots}$ that describes the evolution of $X$ among the states in $S$.
Since we will not be interested in $X$
after exiting $S$, we can simply append an additional state $\delta$ that is
absorbing.  The probability $1 - K(x,S)$ can be
thought of as the probability of jumping from $x$ to the absorbing state $\delta$.

Our primary interest is whether the following sequence of conditional
distributions converges to a proper probability distribution; that is,
whether
\begin{alignat}{3}
	\frac{K^n(x,y)}{K^n(x,S)} &= \pr{X_n = y \given X_n \in S,
X_0 = x}
	&\to \pi_x(y) \tag{YL}
\end{alignat}
as $n \to \infty$ where $x$ and $y$ are in $S$ and $\pi_x$ is a proper probability distribution over $S$.
When (YL) holds, we will say that we have a
\emph{Yaglom limit}, which is named after the author of \cite{Yaglom1947}.

Although, we have formulated this as a discrete time Markov chain $X$, it is quite
common to formulate this as a continuous time Markov chain.
Frequently, we cite results without mentioning whether the results come from a
discrete time or continuous time formulation since it is usually
straightforward to translate continuous time results into the analogous
discrete time result; see Section~3.4 \cite{vanDoornPollettSurvey2013}.

The
review paper \cite{vanDoornPollettSurvey2013} cites applications of these ideas in the areas of
cellular automata, complex systems,
ecology, epidemics, immunology,
medical decision making, physical chemistry,
queues, reliability,
survival analysis, and telecommunications.

To simplify the problem, we will assume that $X_0 \in S$ and that $K$ is irreducible and
aperiodic.  By
\emph{irreducible}, we mean that for any pair of states $x$ and $y$ in $S$, there
exists an $n$ so that $K^n(x,y) > 0$.
By \emph{aperiodic}, we mean that the greatest common divisor of $\set{n > 0 :
	K^n(x,y) > 0}$ is 1 for some pair of states $x$ and $y$, which implies
that it holds for all pairs of states.  The random walk Example \ref{McDF} would be
aperiodic if $r > 0$.

 The potential of $K$ is the generating function
\begin{alignat}{2}
G_{x,y}(z) \equiv \sum_{n \geq 0} K^n(x,y) z^n \label{eqn:GenFun}
\end{alignat}
and $\rho = 1/R$ where $R\geq 1$ is the common radius of convergence of the potential; i.e. independent of $x,y$.
 Let $\zeta = \inf\set{n : X_n \not\in S}$ denote the exit time from $S$, which is
also called the time of absorption.  Clearly,  $K^n(x,S) =
\prs{x}{\zeta >
	n}$ where the subscript $x$ denotes that we are also conditioning on
$X_0 = x$. As remarked on page 405 of \cite{Vere-Jones-Seneta}, if $K$ is strictly substochastic we may assume without loss of generality that absorption is certain; i.e. $P_x(\zeta<\infty)=1$ for all $x\in S$. If not just consider the processes conditioned on being absorbed which
has kernel $K(x,y)g(y)/g(x)$ where $g(x)=P_x(\zeta<\infty)$ is a superharmonic function.

In \cite{McFoley} there is a example of an $R$-transient chain on a countable state space
which has Yaglom limits which depend on the initial state; i.e.
$\lim_{n\to\infty}\frac{K^n(x,y)}{K^N(x,S)}=\pi_x(y)$ where $\pi_x$ belongs to a family of
($\rho$-invariant) probability measures.
The goal of this paper is to generalize this example. This is a somewhat daunting task since
 the remarkable paper by Kesten \cite{Kesten} gives a counterexample which is  similar
 to the example in \cite{McFoley}.
 Kesten's paper \cite{Kesten} does give conditions for a Strong Ratio Limit Property (SRLP) and a Yaglom limit for $R$-transient Markov chains but
 only on $\{0,1,2,3,\ldots\}$ where the Martin boundary only contains one point;
 associated with a unique $\rho$-invariant measure. The proof of convergence depends heavily on this uniqueness.  Our results involve Yaglom limits for
 chains with non-trivial Martin boundary. We restrict to cases where the Jacka-Roberts condition
 holds (see Condition [5] below). This condition fails in Kesten's counterexample.

 The strong connection between Yaglom limits and the space-time entrance boundary was first
discussed in Breyer \cite{Breyer}. When we focus on $R$-transient nearest neighbour Markov chains
on the integers we can  obtain a description of the space-time
Martin entrance boundary which allows us to prove Yaglom limits when the Jacka-Roberts condition holds.

\section{Preliminaries, Conditions and Definitions}
\subsection{Consequences of uniform aperiodicity}
 We use the uniform aperiodicity condition introduced by Kesten \cite{Kesten}  as Condition (1.5)):
\begin{itemize}\label{uniper}
\item[Condition [1]] There exists constants $\delta_1>0$ and $N<\infty$, and for each $i\in S$, there exist integers
$1\leq k_1,\cdots ,k_r\leq N$ (with $k_j=k_j(i)$ and $r=r(i)$) such that $K^{k_s}(i,i)\geq \delta_1$ for $1\leq s\leq r$, and $\mbox{g.c.d.}(k_1,\ldots ,k_r)=1$.
\end{itemize}
As remarked by Kesten \cite{Kesten}, uniformly in $x$, there exists  some $k_0 < \infty$ and $\delta(k) > 0$
independent of $x$ such that $K^k(x, x) \geq \delta(k)$ for $ k \geq k_0$.

\begin{Lem}\label{follow}
Let $K$ be an  irreducible,  kernel $K$ on a countable state space $S$ with spectral radius $\rho$.
Assuming Condition [1]
\begin{eqnarray}
\lim_{n\to\infty}\frac{K^{n+1}(x,y)}{K^{n}(x,y)}&=&\rho,\label{Kesten-lim}
\end{eqnarray}

Again assuming Condition [1], if there exists a $\rho$-excessive probability measure $\mu$, then
\begin{eqnarray}
\lim_{n\to\infty}\frac{K^{n+1}(x,S)}{K^{n}(x,S)}&=&\rho, \label{us-inter-lim}
\end{eqnarray}
which implies that $\lim_{n\to\infty}K^n(x,S)^{1/n}= \rho$.
\end{Lem}

Expression (\ref{Kesten-lim})  is from Lemma 4 in Kesten (1995)
\cite{Kesten} which  holds for both $R$-recurrent and $R$-transient chains.
Equation (\ref{us-inter-lim}) generalizes Lemma 3.1 in \cite{spanish} to include the $R$-transient case.
 In the $R$-recurrent case the existence of $\mu$ is automatic (see Corollary 2 in \cite{Vere-Jones-Seneta})
  so the assumption is that $\mu$ is a finite measure. In the $R$-transient case we can define a $\rho$-excessive
  measure $\mu(x)=G_{x_0,x}(R)$ for any initial point $x_0$ so again the assumption is that $\mu$ is a finite measure (see Condition [A] below).

 The existence of a finite $\rho$-excessive probability $\mu$ implies $$\lim_{n\to\infty}K^n(x,S)^{1/n}= \rho.$$ To prove this we first remark
$K^n(x,S)\geq K^n(x,x)$ and by Theorem A in \cite{Vere-Jones} $\lim_{n\to\infty}K^n(x,x)^{1/n}= \rho$.
Hence $\liminf_{n\to\infty}K^n(x,S)^{1/n}\geq \rho$.  To obtain the opposite inequality observe
$$K^n(x,S)=\rho^n\frac{\mu \overleftarrow{K}^n(x)}{\mu(x)}\leq \rho^n\frac{1}{\mu(x)}
\mbox{ where }\overleftarrow{K}(y,x)=\frac{\mu(x)K(x,y)}{\rho\mu(y)}.$$

Note that \eqref{us-inter-lim} is not always true.   For example,
let $Q$ be the (stochastic) transition matrix of a simple, asymmetric random walk on the
integers with $Q(i, i - 1) = a$, $Q(i, i + 1) = b$, $a+b=1$.  Let $0 < \alpha < 1$.  Then $\alpha Q$ is a
substochastic matrix where the simple, asymmetric random walk is killed with
probability $\alpha$ at each step, and $K = \alpha^2 Q^2$ is the two step
transition matrix.  The transition matrix $K$ restricted to the even integers is
aperiodic and satisfies Condition [1].  Now, $K^n(x,S) = \alpha^{2n}$, but
\begin{alignat*}{2}
	K^n(x,x) &= \alpha^{2n} \binom{2n}{n} a^n b^n \\
	&\sim \alpha^{2n} \frac{(4ab)^n}{\sqrt{\pi n}}.
\end{alignat*}
Thus, $K^n(x,S)^{1/n} \to \alpha^2$, but $K^n(x,x)^{1/n} \to \alpha^2 4ab$.  This does not contradict Lemma~\ref{follow} since $K$ does not
possess a summable, excessive probability measure $\mu$.

\begin{proof}[Proof of Lemma \ref{follow}]
It suffices to extend the powerful proof in Kesten (1995) \cite{Kesten} which as he points out is built on
  the proof of (5) in Theorem 1.1 in \cite{Kingman-Orey} or Theorem 2.1 in \cite{Orey}.
   (\ref{Kesten-lim})  is Lemma 4 in Kesten (1995) \cite{Kesten} which  holds for $R$-recurrent and  $R$-transient chains. The key idea in all these proofs is to represent $K^k=(1-\delta(k))\hat{Q}+\delta(k)I$
   where by Condition [1] $K^k(x,x)\geq 2\delta(k)$ for $k\geq k_0$ uniformly in $x$
   and where $\hat{Q}$ is a transition kernel with excessive measure $\mu$
   having (necessarily positive) spectral radius
   $\hat{\rho}=1/\hat{R}=\frac{\rho^k-\delta(k)}{1-\delta(k)}$.

The extension to (\ref{us-inter-lim}) requires we replace $j$ with $S$ in Kesten's proof. The only
additional requirement is to show (2.16) in \cite{Kesten} holds with $j$ replaced by $S$.
Now $\rho(m,x,S)=\sum_y(\hat{Q})^r(x,y)K^s(y,S)$ where $m=rk+s$, $0\leq s<k$.
Hence $\rho(m,x,S)\leq (\hat{Q})^r(x,S)$ and as above $(\hat{Q})^r(x,S))^{1/r}\to \hat{\rho}$
so $\limsup_{r\to\infty}\rho(rk+s,x,S)\leq \hat{\rho}$.
On the other hand $\rho(m,x,S)\geq \rho(m,x,x)$ and by (2.17) in \cite{Kesten}
$\lim_{r\to\infty}(\rho(rk+s,x,x))^{1/r}=\hat{\rho}$ so
$\liminf_{r\to\infty}(\rho(rk+s,x,S))^{1/r}\geq \hat{\rho}$. This gives (2.16) in \cite{Kesten}.

The rest of the proof follows unchanged and the last statement holds since the the ratio test
is a corollary of the root test.
\end{proof}

We will need the following extension.
\begin{Lem}\label{follow2}
Let $\mu$ be a $\theta$-invariant probability measure where $\theta\in [\rho,1)$.
Let $x_n$ be a sequence  such that $\lim_{n\to\infty}K^n(x_n,y)^{1/n}= \theta$ for some $y$ (and hence all $y$) and such that
and $\lim_{n\to\infty}\mu(x_n)^{1/n}= 1$ then for all $y\in S$, and integers $t$
\begin{eqnarray}
\lim_{n\to\infty}\frac{K^{n+t}(x_n,y)}{K^{n}(x_n,y)}=\theta^t&\mbox{ and }&\lim_{n\to\infty}\frac{K^{n+t}(x_n,S)}{K^{n}(x_n,S)}=\theta^t.\label{us-lim2}
\end{eqnarray}
Moreover $\lim_{n\to\infty}K^n(x_n,S)^{1/n}= \theta$.
\end{Lem}

\begin{proof}
We  prove the first part of (\ref{us-lim2})
by combining elements
of  the proof of (5) in Theorem 1.1 in \cite{Kingman-Orey} or Theorem 2.1 in \cite{Orey}.
 Let $\overleftarrow{K}$ be the associated time reversed kernel; i.e. $\overleftarrow{K}(y,x)=\mu(x)K(x,y)/(\theta\mu(y))$. We remark that $\overleftarrow{K}$ has spectral radius
$1$ and is uniformly aperiodic so there exists a $k_0$ such that
$\overleftarrow{K}^d(x,x)>2\delta(d)>0$ for $d\geq k_0$ uniformly in $x$.

We prove the analogue of (2.14) in \cite{Kesten} by showing:
$$\frac{K^{n+t}(x_n,y)}{\theta^t K^{n}(x_n,y)}
=\frac{\overleftarrow{K}^{n+t}(y,x_n)}{\overleftarrow{K}^{n}(y,x_n)}\to 1.$$
 As in \cite{Kesten} take $d=k_0$ and define $\delta(d)=\delta$ and
$\hat{Q}\equiv \hat{Q}_d=(\overleftarrow{K}^d-\delta)/(1-\delta)$.
Note that $\hat{Q}$ is still irreducible and uniformly aperiodic.

For $n=rd+s$, $0\leq s<d$, define
$\rho(n;y,x_n)=\sum_z(\hat{Q})^r (y,z)\overleftarrow{K}^s(z,x_n)\leq 1$. As in (2.15) in
\cite{Kesten}
\begin{eqnarray}
\overleftarrow{K}^n(y,x_n)&=&
\sum_{\ell=0}^r
\left(\begin{array}{c} r\\\ell\end{array}\right)
\delta^{\ell}(1-\delta)^{r-\ell}\sum_p(\hat{Q})^{r-\ell}(y,p) \overleftarrow{K}^s(y,x_n)\nonumber\\
&=&\sum_{\ell=0}^r B(r,\ell)\rho((r-\ell)d+s;y,x_n)\label{bottom}
\end{eqnarray}
where $B(r,\ell)=\left(\begin{array}{c} r\\\ell\end{array}\right)
\delta^{\ell}(1-\delta)^{r-\ell}$
and
\begin{eqnarray}
\overleftarrow{K}^{n+d}(y,x_n)&=&
\sum_{\ell=0}^{r+1}
B(r+1,\ell)\rho((r+1-\ell)d+s;y,x_n)\nonumber\\
&=&
\sum_{\ell=-1}^{r}
B(r+1,\ell+1)\rho((r-\ell)d+s;y,x_n)\label{top}
\end{eqnarray}

Now, $\overleftarrow{K}^n(y,x_n)= \theta^{-n}\mu(x_n)K^n(x_n,y)/\mu(y)$ and this decays slowly
since $$\lim_{n\to\infty}K^{n}(x_n,y)^{1/n}= \theta\mbox{ and }\lim_{n\to\infty}\mu(x_n)^{1/n}= 1.$$ This allows us to follow (2.2) in \cite{Orey}.
We split the sums in (\ref{top}) and (\ref{bottom}) parts close to the mean and parts a large
deviation away, we throw away the large deviation parts and then show
to show the ratio of the central part (\ref{top}) divided by the central part of (\ref{bottom}) tends to one.

More specifically following \cite{Orey} let $\sum '$ denote summation over $\ell$ satisfying
$|\ell-\delta r|\leq\epsilon r$ while $\sum ''$ denotes summation over $\ell$ satisfying
$|\ell-\delta r|>\epsilon r$. Therefore
\begin{eqnarray}
\lefteqn{\frac{\overleftarrow{K}^{n+d}(y,x_n)}{\overleftarrow{K}^n(y,x_n)}=
\frac{\sum '
B(r+1,\ell+1)\rho((r-\ell)d+s;y,x_n)}{\overleftarrow{K}^n(y,x_n)}}\nonumber\\
& &+\frac{\sum ''B(r+1,\ell+1)\rho((r-\ell)d+s;y,x_n)}
{\overleftarrow{K}^n(y,x_n)}.\label{splitit}
\end{eqnarray}
The numerator of the last term approaches zero at an exponential rate while the denominator decays polynomially as above so the last term is negligible.  We also conclude the numerator of the first term also decays polynomially as does the denominator of the first term.
Now split the denominator of the first term into sums $\sum '$ and $\sum ''$.  For the same reason we may throw away
the sum $\sum ''$.  We conclude
\begin{eqnarray*}
\frac{\overleftarrow{K}^{n+d}(y,x_n)}{\overleftarrow{K}^n(y,x_n)}&\sim&
\frac{\sum ' B(r+1,\ell+1)\rho((r-\ell)d+s;y,x_n)}
{\sum ' B(r,\ell)\rho((r-\ell)d+s;y,x_n)}
\end{eqnarray*}

Now for $|\ell-\delta r|\leq\epsilon r$,
$B(r+1,\ell+1)/B(r,\ell)=(r+1)/(\ell+1)$ is between $\frac{\delta}{\delta+\epsilon}(1+{\cal O}(1/r))$
and $\frac{\delta}{\delta-\epsilon}$. Since $\epsilon$ is arbitrarily small it follows that
 $\overleftarrow{K}^{n+d}(x_n,y)/\overleftarrow{K}^{n}(x_n,y)\to 1$
as $n\to\infty$.

The same argument implies  $\overleftarrow{K}^{n+k}(y,x_n)/\overleftarrow{K}^{n}(y,x_n)\to 1$
as $n\to\infty$  for $k\geq n_0$ or $k\leq -n_0$ so for instance
$\overleftarrow{K}^{n-n_0}(y,x_n)/\overleftarrow{K}^{n}(y,x_n)\to 1$.
Also for $\ell<n_0$,
\begin{eqnarray}
\frac{\overleftarrow{K}^{n+\ell}(y,x_n)}{\overleftarrow{K}^{n}(y,x_n)}&=&
\frac{\overleftarrow{K}^{n-n_0+n_0+\ell}(x_{n},y)}{\overleftarrow{K}^{n}(x_{n},y)}\nonumber\\
&=&\frac{\overleftarrow{K}^{n-n_0+k}(x_{n},y)}{\overleftarrow{K}^{n-n_0}(y,x_n)}
\frac{\overleftarrow{K}^{n-n_0}(y,x_n)}{\overleftarrow{K}^{n}(x_{n},y)}\label{evaluate}
\end{eqnarray}
with $k=n_0+\ell\geq n_0$.
Again using the above argument we see
\begin{eqnarray*}
\frac{\overleftarrow{K}^{n-n_0+k}(x_{n},y)}{\overleftarrow{K}^{n-n_0}(y,x_n)}\to 1
\end{eqnarray*}
so (\ref{evaluate}) tends to one. A similar argument gives for $-n_0<k<0$ gives
\begin{eqnarray*}
\frac{\overleftarrow{K}^{n+k}(y,x_n)}{\overleftarrow{K}^{n}(x_{n},y)}\to 1
\end{eqnarray*}
for all $k$ as $n\to\infty$.

We prove  the second expression in (\ref{us-lim2})
by showing
$$\frac{K^{n+1}(x_n,S)}{\theta K^{n}(x_n,S)}
=\frac{\mu\overleftarrow{K}^{n+1}(x_n)}{\mu\overleftarrow{K}^{n}(x_n)}\to 1.$$
The above steps all hold as long as $\lim_{n\to\infty}K^n(x_n,S)^{1/n}= \theta$ but this follows as above.
\end{proof}

\begin{Lem}\label{trivialtail}
Consider an irreducible Markov chain $Z_n$ on $S$ with kernel $Q$  satisfying Condition [1] then
the tail field and the invariant field of $Z_n$ are equal a.s. with respect to $P_{\nu}$ for any initial probability $\nu$.
\end{Lem}

\begin{proof}
By Theorem 6 in \cite{Cohn} it suffices to show $\sup_{x\in S}\gamma(x)=0$ where
$$\gamma(x)=\lim_{m\to\infty}\sum_{y\in S}|Q^m(x,y)-Q^{m+1}(x,y)|.$$
As in the proof of Lemma \ref{follow}
 take $d=k_0$ and define
$\hat{Q}\equiv \hat{Q}_d=(\overleftarrow{K}^d-\delta)/(1-\delta)$
and for $m=rd+s$, $0\leq s<d$, define
$\rho(m;y)=\sum_{z\in S}\hat{Q}^r(x,z) Q^s(z,y)$. As in (2.15) in
\cite{Kesten}
\begin{eqnarray}
{Q}^m(x,y)&=&
\sum_{\ell=0}^r
\left(\begin{array}{c} r\\\ell\end{array}\right)
\delta^{\ell}(1-\delta)^{r-\ell}\sum_z\hat{Q}^{r-\ell}(x,z) Q^s(z,y)\nonumber\\
&=&\sum_{\ell=0}^r B(r,\ell)\rho((r-\ell)d+s,y)\label{bottomII}
\end{eqnarray}
where $B(r,\ell)=\left(\begin{array}{c} r\\\ell\end{array}\right)
\delta^{\ell}(1-\delta)^{r-\ell}$
and
\begin{eqnarray}
Q^{m+d}(x,y)&=&
\sum_{\ell=0}^{r+1}
B(r+1,\ell)\rho((r+1-\ell)d+s,y)\nonumber\\
&=&
\sum_{\ell=-1}^{r}
B(r+1,\ell+1)\rho((r-\ell)d+s,y)\label{topII}
\end{eqnarray}

Consequently,
\begin{eqnarray*}
\lefteqn{\sum_{y\in S}|Q^m(x,y)-Q^{m+d}(x,y)|}\\
&\leq&\sum_{y\in S}B(r+1,0)\rho((r+1)d+s,y)\\& &+
\sum_{\ell=0}^r|B(r+1,\ell+1)- B(r,\ell)|\sum_{y\in S}
\rho((r-\ell)d+s,y)\\
&\leq &B(r+1,0)+\sum\  '|\frac{B(r+1,\ell+1)}{B(r,\ell)}- 1|B(r,\ell)\\
& &+\sum\  ''|\frac{B(r+1,\ell+1)}{B(r,\ell)}- 1|B(r,\ell)
\end{eqnarray*}
as at (\ref{splitit}). $\sum '$ goes to zero as $m$ and hence $r$ tend to infinity (uniformly in $x$) since
$$\frac{B(r+1,\ell+1)}{B(r,\ell)}=\delta\frac{r+1}{\ell+1}\to 1$$ uniformly for $\ell$ in the range $|\ell-r\delta|\leq \epsilon r.$
The sum of terms $B(r,\ell)$ in $\sum ''$   is exponentially small
and since $B(r+1,\ell+1)/B(r,\ell)=(r+1)/(\ell+1)\leq r+1$ we conclude
$\sum ''$ also tends to zero.
By uniform aperiodicity and the above argument,
\begin{alignat*}{2}
	\sum_{y \in S} \abs{Q^m(x,y) - Q^{m + d}(x,y)} &\to 0.
\end{alignat*}
Hence, the result holds by the triangle inequality.
\end{proof}

\subsection{Tightness}
We now impose
\begin{itemize}\label{bigone}
\item[Condition [2]] $R>1$ and $E_z R^{\zeta}<\infty$ for all $z$.
\item[Condition [3]]  For any $m$, $P_z(\zeta>m)\to 1$ as $z\to\infty$
	(\mbox{ see (1.10) in \cite{Kesten}})\\
\end{itemize}

We will need the following properties:
\begin{itemize}\label{summable}
\item[Condition [A]] $s(z)=\sum_{y\in S} G_{z,y}(R)<\infty$ for one $z$ and hence all $z\in S$
\end{itemize}
\begin{itemize}\label{limitprob}
\item[Condition [B]] $\chi(z,\cdot)=G_{z,\cdot}(R)/s(z)$ is a tight family of probability measures
as $z\to \infty$.
\end{itemize}

\begin{Lem}
If Condition [2] holds then Condition [A] holds.
\end{Lem}

\begin{proof}
Note that
\begin{eqnarray*}
G_{z,S}(R)&=&\sum_{y\in S}\sum_{n \geq 0} K^n(z,y)R^n
=\sum_{n \geq 0} P_z(\zeta>n)R^n\\
&=&\sum_{n=0}^{\infty}\sum_{k=1}^{\infty}\chi\{k>n\}R^nP_z(\zeta=k)\\
&=&\sum_{k=1}^{\infty}\frac{R^k-1}{R-1}P_z(\zeta=k)=\frac{E_zR^{\zeta}-1}{R-1}.
\end{eqnarray*}
Hence $G_{z,S}(R)<\infty$ if Condition [2] holds.
\end{proof}

If Condition [A] holds follow \cite{Dynkin} to construct the entrance boundary taking standard function $\ell(\cdot)=1$
so the $\rho$-Martin entrance kernel is defined by $\chi(z,\cdot)=G_{z,\cdot}(R)/s(z)$.
Notice that the Martin kernel $\chi$ uses reference function $1$ because
we want to deal with $\rho$-invariant probabilities.
Hence $\chi$ is $\rho$-excessive probability in $x$ and would serve as $\mu$
in (\ref{us-inter-lim}).

\begin{Pro}\label{yaglomtight}
If $K$ satisfies Conditions [1],[2] and [3] then $\chi(z,\cdot)$ is a tight family of probabilities; i.e. Condition [B] holds.
\end{Pro}

\begin{proof}
By Condition [3], $E_z R^{\zeta}\to \infty$ as $z\to\infty$.
Hence $E_z R^{\zeta}/(E_z R^{\zeta}-1)$ is bounded by some constant $B$ as a function of $z$.
For any $\epsilon >0$ pick an integer $m$ such that $R^{-m}B<\epsilon/2$.
Next, pick a compact set $C$ such that for $x\in U=C^c$, $P_x(\zeta\leq m)<\epsilon/2$.
\begin{eqnarray*}
\lefteqn{\chi(z,U)=\frac{\sum_{x\in U}\sum_{n=0}^{\infty}R^n K^n(z,x)}{\sum_{n=0}^{\infty}R^n K^n(z,S)}     }\\
&=&\frac{\sum_{x\in U}\sum_{n=0}^{\infty}R^n K^n(z,x)P_x(\zeta>m)}{\sum_{n=0}^{\infty}R^n K^n(z,S)}+
\frac{\sum_{x\in U}\sum_{n=0}^{\infty}R^n K^n(z,x)P_x(\zeta\leq m)}{\sum_{n=0}^{\infty}R^n K^n(z,S)}\\
&\leq&\frac{\sum_{n=0}^{\infty}R^n P_z(\zeta>n+m)}{\sum_{n=0}^{\infty}R^n K^n(z,S)}+ \frac{\sum_{n=0}^{\infty}R^nK^n(z,U)\sup_{x\in U}P_x(\zeta\leq m)}{\sum_{n=0}^{\infty}R^n K^n(z,S)}\\
&\leq&\frac{\sum_{n=0}^{\infty}R^n P_z(\zeta>n+m)}{\sum_{n=0}^{\infty}R^n K^n(z,S)}+\epsilon/2
\end{eqnarray*}

Next,
\begin{eqnarray*}
\sum_{n=0}^{\infty}R^n P_z(\zeta>n+m)&=&\sum_{n=0}^{\infty}R^n \sum_{k=1}^{\infty}\chi\{k>n+m\}P_z(\zeta=k)\\
&=&\sum_{k=m+1}^{\infty}P_z(\zeta=k)\sum_{n=0}^{k-m-1}R^n\\
&=&\sum_{k=m+1}^{\infty}\frac{R^{k-m}-1}{R-1}P_z(\zeta=k)\\
&\leq&R^{-m}\sum_{k=m+1}^{\infty}\frac{R^{k}}{R-1}P_z(\zeta=k)\\
&\leq&R^{-m}\frac{E_z R^{\zeta}}{R-1}.
\end{eqnarray*}
Consequently
\begin{eqnarray*}
\frac{\sum_{n=0}^{\infty}R^n P_z(\zeta>n+m)}{\sum_{n=0}^{\infty}R^n K^n(z,S)}&\leq&R^{-m}\frac{E_z R^{\zeta}}{E_zR^{\zeta}-1}\leq \frac{\epsilon}{2}
\end{eqnarray*}
We conclude that $\chi(z,U)\leq \epsilon$ for all $z$ so the sequence $\chi(z,\cdot)$ is tight.
\end{proof}

\begin{Pro}\label{space-time-tightness}
If Conditions [1,2,3] hold then the family of probabilities $\frac{K^n(x,\cdot)}{K^n(x,S)}$ is tight.
\end{Pro}

\begin{proof}
We again follow the proof of tightness at the end of the example in Kesten (1995). By Condition [2] $\rho<1$. Pick $\epsilon >0$ and  pick $m$ such that
$\rho^m<\epsilon/4$,  So using Condition [1] and Lemma \ref{follow} for $n$ sufficiently large
$K^{n+m}(x,S)/K^n(x,S)<\epsilon/2$.
If Condition [3] holds
we can pick a finite set $C$ such that $P_z(\zeta\leq m) <\epsilon/2$ for $z\in U=C^c$.
Hence,
\begin{eqnarray*}
\frac{K^n(x,U)}{K^n(x,S)}&=&\frac{\sum_{y\in U}K^n(x,y)P_y(\zeta\leq m)}{K^n(x,S)}+\frac{\sum_{y\in U}K^n(x,y)P_y(\zeta > m)}{K^n(x,S)}\\
&\leq&\frac{\sum_{y\in U}K^n(x,y)P_y(\zeta\leq m)}{K^n(x,S)}+\frac{K^{n+m}(x,S)}{K^n(x,S)}\\
&\leq &\epsilon/2+\epsilon/2.
\end{eqnarray*}

The extension to a varying starting points follows by the above argument.
\end{proof}

\subsection{Spatial Martin boundaries}
As in \cite{Dynkin},  the space of $\rho$-harmonic functions is described by the space of exits $B$ inside the  $\rho$-Martin exit boundary $E$; i.e. points $e\in B$ in the completion corresponding to limits $h_e(x)= \lim_{z\to e}G_{x,z}(R)/G_{x_0,z}(R) $
  where $h_e$ is $\rho$-harmonic and minimal.
The $h_e$-transform $\tilde{X}^{h_e}$ with respect to a $\rho$-harmonic function   has a probability transition kernel
$\tilde{K}(x,y)\equiv \tilde{K}^{h_e}(x,y)=RK(x,y)h_e(y)/h_e(x)$
and $P(\tilde{X}_n\to e)=1$
 (see Theorem 11 in \cite{Dynkin}).

Suppose  $z_n\to e$ in the Martin topology.
Let $h$ be some harmonic function.
We say $h$ satisfies the relative Fatou theorem \cite{Doobrelative} if $h(z_n)/h_e(z_n)\to \alpha_e$
where $\alpha_e$ is some constant that does not depend on the sequence.

\begin{Pro} Let $h$ be a $\rho$-harmonic function satisfying the relative
Fatou theorem.  The $h$-tranform $X^h_n$ converges almost surely to $X^h_{\infty}$ taking values
in the Martin boundary. Suppose that $P_{x}(X^h_{\infty}=e)>0$ for all $x\in S$.
Then $\lim_{z_n\to e}P_{z_n}(X^h_{\infty}=e)=1$.
\end{Pro}

\begin{proof} Suppose $U_m$ is any sequence of $\epsilon=1/m$ neighbourhoods of $e$.
\begin{eqnarray*}
P_{z_n}(X^h_{\infty}=e)&=&
\lim_{m\to\infty}\lim_{N\to\infty}P_{z_n}(X^h_{N}\in U_m)\\
&=&\lim_{m\to\infty}\lim_{N\to\infty}E_{z_n}[\frac{h_e(z_n)}{h(z_n)}
\frac{h(X^{h_e}_N)}{h_e(X^{h_e}_N)}\chi\{X^{h_e}_N\in U_m\}]\\
&=&\frac{h_e(z_n)}{h(z_n)}\alpha_e
\end{eqnarray*}
 since $X^{h_e}_N\to e$ with probability one.
 Hence, as $z_n\to e$,
$$P_{z_n}(X^h_{\infty}=e)=\frac{h_e(z_n)}{h(z_n)}\alpha_e\to 1$$
as $z_n\to e$.
\end{proof}

As in \cite{Dynkin} the space of $\rho$-invariant measures is described by  the space of entries $\hat{B}$ in the $\rho$-Martin entrance boundary $\hat{E}$. $b\in \hat{E}$ is a point in the completion of $S$ corresponding to the limit
$\pi_{b}(x)= \lim_{z\to b}\chi(z,x) $ where $\pi_b$ is an extremal $\rho$-invariant measure.

 By Conditions A and B above $\pi_b$ is a probability and by Theorem 11 in \cite{Dynkin}
 any $\rho$-invariant probability measure is a convex combination of the $\{\pi_b:b\in \hat{B}\}$ as in (73) of Theorem 11 in \cite{Dynkin}. In many cases, like Example \ref{twosided-asym}, Conditions A and B aren't necessary because we can calculate the extremals $\pi_b$ explicitly and verify they are probabilities.

The time reversal  $\overleftarrow{X}$ has kernel
$\overleftarrow{K}^{\pi_b}(x,z)=\frac{\pi_b(z)}{\rho\pi_b(x)}K(z,x)$.
The associated potential is denoted
 $\overleftarrow{G}^{\pi_b}(x,z)$ and using $\pi_b$ as a reference measure
 define $\overleftarrow{\nu}(z)=\sum_{y\in S}\pi_b(y)\overleftarrow{G}^{\pi_b}(y,z)$. The associated
 Martin kernel is
 $$\overleftarrow{k}_b(x,z)=\frac{\overleftarrow{G}^{\pi_b}(x,z)}{\overleftarrow{\nu}(z)}=\frac{\chi(z,x)}{\pi_b(x)}$$
 and the associated exit boundary coincides with $\hat{E}$
 (see Theorem 11 in \cite{Dynkin}). Moreover a.s. $P_{\pi_b}$, $\overleftarrow{X}_n\to \overleftarrow{X}_{\infty}=b$ (in the Martin entrance topology) and $\overleftarrow{k}_b(x,b)=1$.

 By Theorem 1 in \cite{Cohn}
 $$1=\overleftarrow{k}_b(x,b)=\frac{dP_x}{dP_{\pi_b}}|_{\cal{I}}\mbox{ $P_{\pi_b}$ a.s.}$$
 where ${\cal I}$ is the invariant $\sigma$ field. If we multiply both sides by the indicator
 of an invariant set $I$ and integrate with respect to $P_{\pi_b}$ we get
 $P_{\pi_b}(I)=P_{x}(I)$ for all $x$. We conclude ${\cal I}$ is trivial a.s. $P_{\pi_b}$.
 In Lemma \ref{trivialtail} we  showed that for chains satisfying Condition [1] the invariant and tail fields
 coincide. Consequently the tail field of $\overleftarrow{X}$ is trivial a.s. $P_{\pi_b}$.

The outline of a proof of Yaglom limits in Subsection \ref{outline}  and the proof in Section \ref{special} requires the following condition.
  \begin{itemize}\label{goodboundary}
\item[Condition [4]]  We assume  $B=\hat{B}$ and both can be identified with the geometric boundary in the sense that if a sequence $y_n\to b$ in the geometric boundary then both
    $h_b(x)= \lim_{y_n\to b}G_{x,y_n}(R)/G_{x_0,y_n}(R) $ and
 $\pi_b(\cdot)= \lim_{y_n\to b}\chi(y_n,\cdot) $.
\end{itemize}

 \subsection{Space-time Martin entrance boundary}
 Consider $K$ as a kernel on space-time $\tilde{S}=S\times (-\infty,\infty)$; i.e.
  $K(x,n;y,n+1)=K(x,y)$.
 Define the space-time Martin kernel
$$k((x,-m);(y,t))=\frac{K^{m+t}(x,y)}{K^m(x,S)}\mbox{ on }S\times (-\infty,0].$$
The space-time Martin entrance boundary is the completion of sequences $(x_i,-m_i)$ to a point $\tilde{b}$ in the space-time Martin entrance boundary $\tilde{N}$
whose space-time Martin kernel $k((x_i,-m_i);(\cdot,\cdot))$ converges to a
corresponding  an invariant measure $\pi(y,t)$ on $\tilde{S}$.

Under weak  conditions we can characterize points on the space-time Martin entrance boundary
as product probabilities:

\begin{Pro}\label{produceit}
Suppose $k((x_i,-m_i);(y,t))\to \pi(y,t)$ for all $y$ and $t$ along a sequence $(x_i,-m_i)$.
Also suppose
  \begin{eqnarray}
\lim_{i\to\infty}\frac{K^{m_i+t}(x_i,y)}{K^{m_i}(x_i,y)}=\theta^t
\mbox{ for all }y,t\label{oreysequence}
\end{eqnarray}
and the probabilities $K^{m_i}(x_i,\cdot)/K^{m_i}(x_i,S)$ are tight. Then  $K^{m_i}(x_{i},\cdot)/K^{m_i}(x_{i},S)$ converges weakly to a $\theta$-invariant probability $\pi$ and $$k((x_i,-m_i);(y,t))\to \pi(y,t)=\pi(y)\theta^t.$$
\end{Pro}
In the case $\theta=\rho$ we can ensure tightness by assuming Conditions [1,2,3] by Proposition \ref{space-time-tightness}. Lemma \ref{follow2} gives criteria for
(\ref{oreysequence}) to hold.  Of course if $x_i=x$ for all $i$ then
(\ref{oreysequence}) holds from Lemma \ref{follow}.

\begin{proof}
 Let
 $\pi(\cdot)=\pi(\cdot,0)$. $\pi(y,0)=\lim_{i\to\infty}K^{m_i}(x_{m_i},y)/K^{m_i}(x_{m_i},S)$
is a probability measure by tightness.
Note
\begin{eqnarray*}
\frac{\pi(y,t+1)}{\pi(y,t)}&=&\lim_{i\to\infty}\frac{k((x_{m_i},-m_i);(y,t+1))}{k((x_{m_i},-m_i);(y,t))}\\
&=&\lim_{i\to\infty}\frac{K^{m_i+t+1}(x_{m_i},y)}{K^{m_i+t}(x_{m_i},y)}=\theta.
\end{eqnarray*}
Hence $\pi(y,t)=\pi(y,0)\theta^t\equiv \theta^t\pi(y)$ where $\pi$ is $\theta$-invariant for $K$ (in space) because $\pi(y,t)$
is invariant for $K$ (in space-time).
\end{proof}

\begin{Pro}
The minimal space-time invariant measures are of product form $\pi(y)\theta^{t}$ where $\pi$ is a minimal
$\theta$-invariant measure for $K$.
\end{Pro}

\begin{proof} Use the argument in Theorem 3.1 in \cite{Lamperti-Snell} (also see \cite{Molchanov}). Suppose $\pi(y,t)$ is an
extremal invariant measure on space-time. Using Condition [1] there exists
a $k_0$ and  $\delta(k)>0$ such that $K^{k}(x,x)\geq \delta(k)$ for all $k\geq k_0$ and all $x$. Pick $k=k_0$.
Define $\pi^{\prime}(x,t)=\pi(x,t+k)$. $\pi^{\prime}$  is also minimal
for suppose $\alpha^{\prime}$ is invariant and $\alpha^{\prime}(x,t)\leq \pi^{\prime}(x,t)$ for all $x$ and $t$ then  $\alpha(x,t)=\alpha^{\prime}(x,t-k)$
is also invariant and $\alpha(y,t)\leq \pi(x,t)$ for all $x$ and $t$.
Hence $\alpha=c\pi$ so $\alpha^{\prime}(x,t)=c\pi^{\prime}(x,t)$.

Next
\begin{eqnarray*}
\pi^{\prime}(x,t)&=&\pi(x,t+k)=\sum_{z\in S}\pi(z,t)K^{k}(z,x)\\
&\geq &\pi(x,t) K^{k}(x,x)\geq \delta \pi(x,t).
\end{eqnarray*}
But $\pi^{\prime}$ is extremal so $\pi(x,t)=c_k\pi^{\prime}(x,t)=c_k\pi(x,t+k)$
for all $x,t$. Hence $\pi(x,t)=c_k^m\pi(x,t+mk)$ for all $m$. But this works
for all $k\geq k_0$ so
$$\pi(x,t+1)=c_k^m\pi(x,t+1+mk)=c_k^m\pi(x,t+m_1k_1)$$
for some other $k_1\geq k_0$ and some other $m_1$.
Hence $\pi(x,t+1)=c_k^m c_{k_1}^{-m_1}\pi(x,t)$; i.e. $\pi(x,t+1)=\theta \pi(x,t)$
for all $x,t$. It follows that $\pi(x,t)=\theta^t \pi(x,0)=\theta^t \pi(x)$
and moreover $\pi$ is $\theta$-invariant.

If $\alpha(x)\leq \pi(x)$ for all $x$ then on space time the measure $\theta^t\alpha(x)$ is invariant and
$\theta^t\alpha(x)\leq \theta^t\pi(x)=\pi(x,t)$. But $\pi(x,t)$ is extremal
on space-time so $\theta^t\alpha(x)= c\theta^t\pi(x)$; i.e. $\alpha(x)=\pi(x)$
for all $x$. We conclude $\pi(x)$ is extremal.
\end{proof}

We will later establish which of the space-time
Martin entrance boundary points correspond to minimal measures.

To understand the complexity of the space-time entrance boundary
suppose Conditions [1,2,3] hold and
 $x_n\to b\in \hat{E}$; i.e.
$\pi_b(\cdot)= \lim_{x_n\to b}\chi(x_n,\cdot)$.
By Proposition \ref{space-time-tightness} the
 probabilities $K^{n}(x_n,\cdot)/K^{n}(x_n,S)$ are tight  hence
 $\pi_b(y)$ is a probability. Further suppose
$\lim K^{n}(x_n,y)^{1/n}= \rho$
and $\pi_b(x_n)^{1/n}\to 1$. We have by Lemma \ref{follow} that
$$\lim_{n\to\infty}\frac{K^{n+t}(x_n,y)}{K^{n}(x_n,y)}=\rho\mbox{ for all $y$}.$$
Now take a subsequence $n_i$ such that $(x_{n_i},-n_i)$ converges to a point $\tilde{b}$ in the space-time
entrance boundary;
i.e.
$$k((x_{n_i},-n_i);(y,t))=\frac{K^{n_i+t}(x_{n_i},y)}{K^{n_i}(x_{n_i},S)}\to \alpha(y,t)\mbox{ for all $y$ and $t$}.$$
where $\alpha$ is invariant for $K$ on space-time. Then by Proposition \ref{produceit} $\alpha(y,t)=\alpha(y)\rho^t$ where $\alpha$ is $\rho$-invariant.

The question is whether $\alpha=\pi_b$. Surprisingly the answer is no in general
as will be seen in \cite{Jaime}. One might conjecture that this holds if
the Jacka-Roberts Condition [5] holds and the associated $\hat{h}$-transform
$X^{\hat{h}}$ has the property $P_x(X^{\hat{h}}=b)>0$ then $\alpha=\pi_b$
and we could then conclude
$$\lim_{n\to\infty}k((x_n,-n);(y.t))
=\lim_{n\to\infty}\frac{K^{n+t}(x_{n},y)}{K^{n}(x_{n},S)}=\pi(y)\rho^t.$$
Such a result would allow us to prove Yaglom limits in general but
so far we only have this result for nearest neighbour random walks on the integers (see Section \ref{special}).

\subsection{Orey paths}
We now turn our attention to $\overleftarrow{X}$.
 Consider Theorem 1.4 in Chapter 3 of \cite{Orey} which we modify slightly.
 \begin{Thm}\label{OreyCnvgThm1}
 Consider a  Markov chain $Z_m$ defined on $(\Omega,{\mathcal F},P_{\alpha})$ taking values in a countable state space $S$ with kernel $Q$ and initial probability distribution $\alpha$. Let $S_m=\{z:\alpha Q^m(z)>0\}$.
 Define
 $$h_m(z)=\frac{\beta Q^{m+d}(z)}{\alpha Q^m(z)}, m\geq \max\{0,-d \},
 h_m(z)=0\mbox{ if }\alpha Q^m(z)=0$$
 where $\beta$ is any probability on $S$.  Then $h_m(Z_m)$ is a backward supermartingale
 with respect to the filtration ${\cal F}^m=\sigma(Z_m,Z_{m+1},\ldots)$ for the measure $P_{\alpha}$. Also $E_{\alpha}h_m(Z_m)\leq 1$.

 Moreover, if  $\beta Q^{m+d}(S_m)\to 1$
   as $m\to\infty$ then $E_{\alpha}h_m(Z_m)\to 1$ and
 $$\lim_{m\to\infty}\frac{\beta Q^{m+d}(Z_m)}{\alpha Q^m(Z_m)}= \frac{dP_{\beta}}{dP_{\alpha}}|_{{\cal I}}$$
 where   $\frac{dP_{\beta}}{dP_{\alpha}}|_{{\cal I}}$ is the Radon-Nikodym derivative
 of the measure $P_{\beta}$ by $P_{\alpha}$ where both are restricted to
 the invariant $\sigma$-algebra ${\cal I}$.
 If $Z_m$ has trivial tail field relative
 to the measure $P_{\alpha}$
 then $h_m(Z_m)$ converges almost surely to $1$ with respect $P_{\alpha}$.
 \end{Thm}
We use the above theorem with $\alpha$ being a $\theta$-invariant probability
so $S_m=S$ and the requirement $\beta Q^{m+d}(S_m)\to 1$ is automatically true.
It is not always true however.
 Suppose $S=\{1,2,3,\ldots \}$ and $Q(x,x)=1/2^x$ and $Q(x,x+1)=1-1/2^x$ and $\alpha=\beta=\delta_{1}$. Then $S_m=\{1,2,\dots ,m\}$ but
 $$\beta Q^{m+1}(\{m+1\})=\prod_{x=1}^m(1-2^{-x})\geq
 \exp(-\frac{3}{2}\sum_{x=1}^m 2^{-x})>\exp(-3/2)$$
 since $-3z/2<\log(1-z)$ for $0<z\leq 0.58$. Hence $\beta Q^{m+1}(S_m)$
 does not converge to $1$.

 \begin{proof}
 If $\alpha Q^m(z)>0$,
\begin{eqnarray*}
 \lefteqn{E_{\alpha}\left[\frac{\beta Q^{m+d}(Z_m)}{\alpha Q^m(Z_m)}|Z_{m+1}=y\right]
 =\sum_{x\in S}\left[\frac{\alpha Q^m(x)\beta Q^{m+d}(x)Q(x,y)}{\alpha Q^m(x)\cdot \alpha Q^{m+1}(y)}\right]}\\
 &=&\sum_{x\in S_m}\frac{\beta Q^{m+d}(x)Q(x,y)}{\alpha Q^{m+1}(y)}
 \leq \frac{\beta Q^{m+d+1}(y)}{\alpha Q^{m+1}(y)}=h_{m+1}(y).
 \end{eqnarray*}
  Thus $h_m(Z_m)$ is a positive backward supermartingale with respect to $\sigma(Z_m,Z_{m+1},\ldots)$.
 Moreover,
\begin{eqnarray*}
\lefteqn{E_{\alpha}\left[h_m(Z_m)\right]=\sum_{x\in S}\alpha Q^m(x)\frac{\beta Q^{m+d}(x)}{\alpha Q^m(x)}}\\
&=&\sum_{x\in S_m}\beta Q^{m+d}(x)\leq 1.
\end{eqnarray*}

By the backward martingale theorem $h_m(Z_m)$ is uniformly integrable,
 $h_m(Z_m)$ converges almost surely and in $L^1(P_{\alpha})$ to $H$ and $EH = \lim_{m\to\infty}E_{\alpha}h_m(Z_m)$.
 If $\beta Q^{m+d}(S_m)\to 1$ the above shows
 $E_{\alpha}\left[h_m(Z_m)\right]\to 1$.
 Consequently $EH=1$. $H$  is measurable with respect to the invariant field
 so if the invariant field is trivial then $H=1$.

 Let $A$ be an invariant event. Then
 $$\int_A\frac{dP_{\beta}}{dP_{\alpha}}|_{{\cal I}}dP_{\alpha}
 =P_{\beta}(A).$$
 On the other hand
 \begin{eqnarray*}
\lefteqn{ E_{\alpha}\left[\frac{\beta Q^{m+d}(Z_m)}{\alpha Q^m(Z_m)}\chi_{A}\right]
 =E_{\alpha}\left[\frac{\beta Q^{m+d}(Z_m)}{\alpha Q^m(Z_m)}\chi_{\theta^{-m}A}\right]}\\
 &=&\sum_{x\in S_m}\left[\alpha Q^m(x)\frac{\beta Q^{m+d}(x)P_{x}(A)}
 {\alpha Q^m(x)}\right]
 =\sum_{x\in S_m}\beta Q^{m+d}(x)P_{x}(A)\\
 &=&P_{\beta}[\{Z_{m+d}\in S_m\}\theta^{-(m+d)}A]\to P_{\beta}(A)
 \end{eqnarray*}
where we used  $A=\theta^{-m}A=\theta^{-(m+d)}A$ and $S_m\uparrow S$.
Since the above holds for all invariant events $A$ we have our result.
\end{proof}
We see that Orey's theorem is the space-time analogue of Abrahamse's
\cite{Abrahamse}.

We now  apply Orey's Theorem \ref{OreyCnvgThm1} to the time reversal $\overleftarrow{X}^{\pi_b}\equiv \overleftarrow{X}$  of $X$ with respect to $\alpha=\pi_b$ where $\pi_b$ is an extremal $\theta$-invariant probability.
The tail field of $\overleftarrow{X}$
is trivial w.r.t. $P_{\pi_b}$.
By Theorem \ref{OreyCnvgThm1}, as $m\to\infty$, with $\beta=\delta_y$,
 \begin{eqnarray}\label{first}
 \frac{\overleftarrow{K}_{\pi_b}^{m+d}(y,\overleftarrow{X}_m)}{\sum_{z\in S}\pi_b(z)\overleftarrow{K}_{\pi_b}^m(z,\overleftarrow{X}_m)}\to 1\mbox{ a.s.} P_{\pi_b}
 \end{eqnarray}
Multiply the numerator of (\ref{first}) by $\pi_b(y)$ and use time reversal we have, almost surely
$P_{\pi_b}$, as $m\to\infty$,
\begin{eqnarray}\label{escape1}
\frac{K^{m+d}(\overleftarrow{X}_m,y)}{K^m(\overleftarrow{X}_m,S)}&\to& \theta^d\pi_b(y);
\end{eqnarray}
i.e. Yaglom limits hold along the trajectory $(\overleftarrow{X}_m,-m)$.

Dividing (\ref{escape1}) with $d=1$ by (\ref{escape1}) with $d=0$ gives $K^{m+1}(\overleftarrow{X}_m,y)/K^{m}(\overleftarrow{X}_m,y)\to\theta$.
Taking $d=1$ and $\beta=\pi_b$ we also get
$K^{m+1}(\overleftarrow{X}_{m},S)/K^{m}(\overleftarrow{X}_{m},S)\to\theta$.

We can reinterpret (\ref{escape1}) as a description of the space-time entrance boundary.
 In this context $(\overleftarrow{X}_{m},-m)$ converges
to a point in the space-time entrance boundary  corresponding to the invariant measure $\theta^t\pi_b(y)$.
In other word, along the Orey path $(\overleftarrow{X}_{m},-m)$
$$k(\overleftarrow{X}_{m},-m;y,t)\to \theta^t\pi_b(y).$$

Next apply Orey's Theorem \ref{OreyCnvgThm1} to the time reversal $\overleftarrow{X}$ with respect
to a $\rho$-invariant measure $\alpha$ having kernel $\overleftarrow{K}_{\alpha}$. Under Condition [1] the tail and invariant $\sigma$-fields are the same and the
invariant field is composed of the events of the form
$\overleftarrow{X}_{\infty}\in B$ where $B$ is a set in the Martin exit boundary.
Suppose $P_{\alpha}(\overleftarrow{X}_{\infty}=b)>0$ for some point $b$
in the exit boundary then on the set $\{\overleftarrow{X}_{\infty}=b\}$
\begin{eqnarray*}
 \frac{\overleftarrow{K}_{\alpha}^m(y,\overleftarrow{X}_m)}{ \alpha\overleftarrow{K}_{\alpha}^m(\overleftarrow{X}_m)}\to
 \frac{P_{y}(\overleftarrow{X}_{\infty}=b)}
 {P_{\alpha}(\overleftarrow{X}_{\infty}=b)}\mbox{ a.s.} P_{\alpha};
 \end{eqnarray*}
i.e.
\begin{eqnarray*}
 \frac{K^m(\overleftarrow{X}_m,y)}{ K^m(\overleftarrow{X}_m,S)}\to
 \alpha(y)\frac{P_{y}(\overleftarrow{X}_{\infty}=b)}
 {P_{\alpha}(\overleftarrow{X}_{\infty}=b)}\mbox{ a.s.} P_{\alpha}
 \end{eqnarray*}
 on the set $\{\overleftarrow{X}_{\infty}=b\}$.

 \subsection{The Jacka-Roberts condition}
By Lemma 2.3 in \cite{Jacka-Roberts}  if
 for some $x_0$,
\begin{eqnarray}
 P_x(\zeta>N-k)/P_{x_0}(\zeta>N)
\mbox{ converges for every $x\in S$ and $k\geq 0$}\label{asymp2}
\end{eqnarray}
then the limit is of the form
$R^k \hat{h}(x)/\hat{h}(x_0)$ where $h$ is a $\rho$-subinvariant function; i.e. Condition [5] below holds.
\begin{itemize}\label{asymp}
\item[Condition [5]]  For some $x_0$ and  for every $x\in S$ and $k\geq 0$
\begin{eqnarray}
\frac{ P_x(\zeta>n-k)}{P_{x_0}(\zeta>n)}&\to& R^k \frac{\hat{h}(x)}{\hat{h}(x_0)}
\mbox{ where $\hat{h}$ is $\rho$-invariant.}\label{asymp1}
\end{eqnarray}
\end{itemize}

Moreover Jacka and Roberts  showed that  the measure $P^N_x$ defined by $\tilde{X}$ started at $(x,-N)$  converges to $P^{\hat{h}}_x$ as $N\to\infty$ iff Condition [5] holds.
Moreover $P^{\hat{h}}_x$ is the measure derived from the $\hat{h}$-transformed Markov chain $X^{\hat{h}}_n; n=0,\ldots, \infty$ with kernel
$K_{\hat{h}}(x,y)=R K(x,y)\hat{h}(y)/\hat{h}(x)$. $X^{\hat{h}}_n$ is the chain conditioned to live forever.
Assuming Condition [5] holds and since $X$
is $R$-transient then $X^{\hat{h}}_n$ is transient. By Theorem 4 in \cite{Dynkin} $X^{\hat{h}}_n\to X^{\hat{h}}_{\infty}$
(in the Martin exit topology) where $X^{\hat{h}}_{\infty}$ is a random variable with support on  $B$, the space of
exits in the $\rho$-Martin exit boundary.
 Let $\mu^{\hat{h}}_x$ be the distribution of $X^{\hat{h}}_{\infty}$
when $X^{\hat{h}}_0=x$.



\subsection{Yaglom limits as points in the space-time Martin boundary}\label{ST}
One way to view Yaglom limits is to define $H(x,-n)=R^nK^n(x,S)$ and remark as in \cite{Doob}  and \cite{Breyer}
that $H$ is space-time $\rho$-harmonic on $S\times (-\infty,0]$; i.e.
 $$H(x,-(n+1))=R\sum_y K(x,-(n+1);y,-n)H(y,-n).$$
 Note that $H(x,-n)=R^n P_x(\zeta>n)$.
From Condition [5] follows that $H(x,-N+k)/H(x_0,-N)\to \hat{h}(x)/\hat{h}(x_0)$.
Condition [5] is not satisfied by Kesten's $R$-transient counterexample and is vital to our proof in
the $R$-transient case.

Now define the $H$-transform of this space time chain:
\begin{eqnarray*}
\tilde{K}(x,-(n+1);y,-n)&=&K(x,-(n+1);y,-n)\frac{H(y,-n)}{\rho H(x,-(n+1))}\\
&=&K(x,y)\frac{R H(y,-n)}{H(x,-(n+1))}.
\end{eqnarray*}
Now notice
\begin{eqnarray*}
\lefteqn{\tilde{K}(x,-N;y,0)}\\
&=&\sum_{x_1,x_2,\ldots x_{N-1}}K(x,x_{N-1})\frac{R H(x_{N-1},-(N-1))}{H(x,-N)}\\
& &\cdot K(x_{N-1},x_{N-2})\frac{R H(x_{N-2},-(N-2))}{H(x,-(N-1))}\cdots
\cdots K(x_1,y)\cdots
\frac{R H(y,0)}{H(x_1,-1)}\\
&=&\frac{R^N K^N(x,y)}{H(x,-N)}=\frac{K^N(x,y)}{K^N(x,S)}
\end{eqnarray*}
using telescoping and the fact that $H(y,0)=1$ for all $y$.
We see we can interpret the Yaglom ratio as a  nonhomogeneous probability transition kernel $\tilde{K}(x,-N;y,0)$.

Also define $N=n+m$ so
\begin{eqnarray*}
\tilde{K}(x,-N;y,-m)&=&K^n(x,y)\frac{H(y,-m)}{\rho^n H(x,-N)}=K^n(x,y)\frac{K^m(y,S))}{K^N(x,S)}\\
&=&\frac{P_x(X_n=y,\zeta>N)}{P_x(\zeta>N)}.
\end{eqnarray*}

If $\pi$ is a $\rho$-invariant measure then $H\cdot \pi$ is a  left invariant measure for $\tilde{K}$:
\begin{eqnarray*}
\lefteqn{\sum_x \pi(x)H(x,-(n+1))\tilde{K}(x,-(n+1);y,-n)}\\
&=&\sum_x \pi(x)K(x,y)R H(y,-n)
=\pi(y)H(y,-n).
\end{eqnarray*}
Note that the time reversal with respect to $H\cdot \pi$ is
\begin{eqnarray*}
\overleftarrow{K}(y,-n;x,-(n+1))&=&\frac{\pi(x)H(x,-(n+1)}{\pi(y)H(y,-n)}\tilde{K}(x,-(n+1);y,-n)\\
&=&\frac{R\pi(x)}{\pi(y)}K(x,y)=\overleftarrow{K}(y,x);
\end{eqnarray*}
i.e. the time reversal of $\tilde{K}$ w.r.t. $H\cdot \pi$ has the same transition probabilities as the time reversal of $K$ with respect to $\pi$.

\subsection{Killing at one point}
\begin{itemize}\label{smallone}
\item[Condition [6]] Killing only occurs at one point $x_0$ with probability $\kappa$.
\end{itemize}

\begin{Lem}\label{onekill}
If Conditions [1,2,3,4] hold then
$$G_{z,x_0}(R)=\frac{1}{\kappa}((1-\rho)G_{z,S}(R)+\rho)\mbox{ and }\lim_{n\to\infty}\frac{K^n(x,x_0)}{K^n(x,S)}\to\frac{1-\rho}{\kappa};$$
\end{Lem}
\begin{proof}
If killing occurs only at $x_0$ with probability $\kappa$ then
$$K^{n+1}(z,S)=K^n(z,S)-K^n(z,x_0)\kappa.$$
Multiplying by $R^n$ and summing on $n$ from $0$ to $\infty$ we get
$$R^{-1}(G_{z,S}(R)-1)=G_{z,S}(R)-G_{z,x_0}(R)\kappa$$ which yields the first result.

Alternatively, dividing by $K^n(x,S)$  and using Condition [1] and Lemma \ref{follow} we get
$$\lim_{n\to\infty}\frac{K^n(x,x_0)}{K^n(x,S)}\to\frac{1-\rho}{\kappa}.$$
\end{proof}
Taking subsequences of $K^{n}(x,\cdot)/K^{n}(x,S)$ we can find quasistationary limits.
If one such limit give $\psi$ we see $\psi(0)=(1-\rho)/\kappa$.

Moreover, with local killing only at $x_0$,
\begin{eqnarray}\label{localkill}
\frac{K^n(x,x_0)}{K^n(y,x_0)}&=&
\frac{K^n(x,S)}{K^n(y,S)}\frac{\left(1-K^{n+1}(x,S)/K^n(x,S)\right)}
{\left(1-K^{n+1}(y,S)/K^n(y,S)\right)}\sim \frac{K^n(x,S)}{K^n(y,S)}
\end{eqnarray}
where we used  Condition [1].
If either of the above limits exist then the Jacka-Roberts Condition [5]
holds and the limit is
$\hat{h}(x)/\hat{h}(y)$.
Hence, $\frac{h_{x_0}(x)}{h_{x_0}(y)}=\frac{\hat{h}(x)}{\hat{h}(y)}$.
This was observed by calculation in the example in \cite{McFoley}.

\subsection{Outline of a proof}\label{outline}

  Starting from $(x,-N)$, the process $\tilde{X}$ behaves like $X^{\hat{h}}$;
i.e. converges to $b$ with probability $\mu_x^{\hat{h}}(b)$.  We can try to couple $X^{\hat{h}}$ at a time $\tau_N$ with a point on
any fixed Orey path $D_b=\{(\overleftarrow{X}_m,-m):m\geq 0\}$ given by the time reversal with respect to $\pi_b$. Note along $(y_m,-m)\in D_b$, $y_m\to b$ and
 $K^m(y_m,\cdot)/K^m(y_m,S)\to \pi_b(\cdot)$.

If we can show this then
\begin{eqnarray*}
\lefteqn{\tilde{K}^N(x,-N;y,0)}\\
&=&\sum_{z}\sum_{k=1}^N P_{(x,-N)}(\tilde{X}_k=z,\tau_N=k)P_{(z,k-N)}(\tilde{X}_{N-k}=y)\\
&\approx&\sum_{z}\sum_{k=1}^N P_{(x,-N)}(X^{\hat{h}}_k=z,\tau_N=k)P_{(z,k-N)}(\tilde{X}_{N-k}=y)\\
&\approx&\sum_{b\in E} P_{x}(X^{\hat{h}}\to b)\pi_b(y)\\
&\approx& \sum_{b\in B} \mu_x^{\hat{h}}(b)\pi_b(y)=\pi_x(y).
\end{eqnarray*}
Unfortunately we can't hope that $\tilde{X}$ starting from $(x,-N)$ will follow $X^{\hat{h}}$ up to the intersection with the path
$(\overleftarrow{X}_m,-m)$. We need to know that other sequences $(x_m,-m)$ such that $x_m\to b$ slowly enough
are also Orey paths. However it does seem there may exist an elegant general coupling proof
along these lines; perhaps inside the Martin compactification of space-time. We couldn't do it and in the next section give the proof
with the assumption $S=\mathbb{Z}$ and that transitions are nearest neighbour.
This is undoubtedly much too strong.

\section{Restriction to nearest neighbour random walks}\label{special}
We now assume $S=\mathbb{Z}$ and that transitions are nearest neighbour. Assume the killing set is at a single point $x_0$ which can be taken to be $0$ so Condition [6] holds.  Let $K(x,x+1)=p_x>0$ and $K(x,x)=r_x$ and $K(x,x-1)=q_x>0$ so the chain is
irreducible. We assume $p_x,r_x,q_x$ are such that the chain is $R$-transient and Conditions [1,2,3] hold.

Finally we add a condition which is much too strong but simplifies our proof.
\begin{itemize}\label{dumb}
\item[Condition [7]] $r_x\geq 1/2$ for all $x$.
\end{itemize}
For any nearest neighbour random walk $W$ with kernel $Q$ Condition [7]
implies that $Q^n(x,\cdot)$ is stochastically smaller than $Q^n(x+1,\cdot)$
for all $x$. To see this take two copies of the chain $W_k$ and $W^*_k$
starting from $x$ and $x+1$ respectively
which evolve independently except when they are one apart. When this happens
$W_k=z$ and $W^*_k=z+1$ for some $z$. On the next step we couple the chains together at $z$
with probability $q_{z+1}$ (which is necessarily less than or equal to $s_z$)
or at $z+1$ with probability $p_z$ (which is necessarily less  than or equal to $s_{z+1}$).
The paths of the resulting coupled chains are always separated by $1$ or else
are coupled together so stochastic monotonicity holds.

Under Condition [2] the Martin entrance kernel $\chi(x,y):=\frac{G_{x,y}(R)}{G_{x,S}(R)}$ exists since Condition [A] holds where we have used the reference function $1$. Using Condition [3] Condition [B] holds
so as $x_n$ tends to $+\infty$ or $-\infty$ the sequence of probability measures $\chi(x_n,\cdot)$
is tight.
Hence the cone of $\rho$-invariant probabilities cannot be empty.
A point $\alpha$ in the Martin boundary
is either in $\{+\infty\}$ which is the Martin-closure of sequences $\{x_n\}$ which tend to $+\infty$  or in $\{-\infty\}$ which is the Martin-closure of sequences $\{x_n\}$ which tend to $+\infty$. Neither $\{+\infty\}$ nor $\{-\infty\}$ can be empty.

We could equally well have defined the entrance kernel
$$\chi_0(x,y):=\frac{G_{x,y}(R)}{G_{x,x_0}(R)}.$$
By Lemma \ref{onekill}, $G_{x,x_0}(R)/G_{x,S}(R)\to (1-\rho)/\kappa$ as $x$ goes to plus
or minus infinity (which forces $G_{x,S}(R)\to\infty$). Hence the $\rho$-Martin entrance boundary
of $\chi_0$ is the same as that of $\chi$ up to a multiple.

Suppose $\alpha$ is a minimal point in $\{+\infty\}$ associated with the extremal $\rho$-invariant
probability $\pi_{\alpha}$. Let $\overleftarrow{X}^{\alpha}$ represent the time reversal with respect $\pi_{\alpha}$ to  which converges
almost surely to $\alpha$ in the Martin entrance topology;
i.e. $$G_R(x_n,y)/G_R(x_n,0)\to \pi_{\alpha}(y)/\pi_{\alpha}(0).$$ Moreover the tail field is trivial
so $\overleftarrow{X}^{\alpha}$ is a $1$-transient Markov chain which diverges almost surely to either plus or minus infinity in the point set topology.
Suppose the latter is true then, as $x_n\to\infty$,
\begin{eqnarray*}
\frac{G_R(x_n,y)}{G_R(x_n,0)}&=&
\frac{\pi_{\alpha}(y)}{\pi_{\alpha}(0)}\frac{\overleftarrow{G}_{\alpha}(y,x_n)}
{\overleftarrow{G}_{\alpha}(0,x_n)}
=\frac{\pi_{\alpha}(y)}{\pi_{\alpha}(0)}\frac{\overleftarrow{A}_{\alpha}(y,x_n)}
{\overleftarrow{A}_{\alpha}(0,x_n)}
\end{eqnarray*}
where $\overleftarrow{A}_{\alpha}(y,x)$ is  the probability $\overleftarrow{X}^{\alpha}$ reaches $x$ from $y$.
If $\overleftarrow{X}^{\alpha}$ were transient to $-\infty$ then for $y<0$, $\overleftarrow{A}_{\alpha}(y,x_n)/\overleftarrow{A}_{\alpha}(0,x_n)$ can't go to $1$
because there is a nonzero probability $\overleftarrow{X}^{\alpha}$ leaves $y$ and diverges to $-\infty$ before hitting $0$.
This is a contradiction.
We conclude $\overleftarrow{X}^{\alpha}$ diverges to $+\infty$ almost surely.

Now consider another sequence $z_n\to \beta\in \{+\infty\}$. Therefore, for $y<0$,
\begin{eqnarray*}
\frac{G_R(z_n,y)}{G_R(z_n,0)}&=&\frac{\pi_{\alpha}(y)}{\pi_{\alpha}(0)}
\frac{\overleftarrow{G}_{\alpha}(y,z_n)}{\overleftarrow{G}_{\alpha}(0,z_n)}\\
&=&\frac{\pi_{\alpha}(y)}{\pi_{\alpha}(0)}\sum_{k=1}^{\infty}f_{y0}(k)
\frac{\overleftarrow{G}_{\alpha}(0,z_n)}{\overleftarrow{G}_{\alpha}(0,z_n)}
=\frac{\pi_{\alpha}(y)}{\pi_{\alpha}(0)}
\end{eqnarray*}
where $f_{x0}(k)$ is the probability the first passage time from $y$ to $0$ by $X^{\alpha}$ is $k$ and
since $X^{\alpha}\to\infty$, $\sum_{k=1}^{\infty}f^h_{x0}(k)=1$.
A similar result holds if $y\geq 0$.
We conclude $\{+\infty\}$ consists of at most a single minimal point associated with an extremal
we denote by $\pi^0_{+\infty}$. Similarly $\{-\infty\}$ consists of a single point associated with $\pi^0_{-\infty}$.

Note $\pi^0_{-\infty}$ can't be $\pi^0_{+\infty}$. Consider a sequence $w_{n}\to\{-\infty\}$; i.e.
$$w_n\to -\infty\mbox{ and } \frac{G_R(w_n,y)}{G_R(w_n,0)}\to \frac{\pi^0_{-\infty}(y)}{\pi^0_{-\infty}(0)}.$$
Taking $\alpha=\pi^{0}_{+\infty}$,
$$\frac{G_R(w_n,y)}{G_R(w_n,0)}=\frac{\alpha(y)}{\alpha(0)}
\frac{\overleftarrow{A}_{\alpha}(y,w_n)}{\overleftarrow{A}_{\alpha}(0,w_n)}$$
 and $X^{\alpha}\to +\infty$ so if $y>0$ $\overleftarrow{A}_{\alpha}(y,w_n)=\overleftarrow{A}_{\alpha}(y,0)\overleftarrow{A}_{\alpha}(0,w_n)$
and $\overleftarrow{A}_{\alpha}(y,0)$ can't equal $1$ because $X^{\alpha}$ is $1$-transient
to $+\infty$. Hence $G_R(w_n,y)/G_R(w_n,0)$ does not converge to $\pi_{+\infty}(y)/\pi_{+\infty}(0)$
so $\pi_{+\infty}\neq \pi_{-\infty}$.

Using Orey's Theorem we have shown that almost surely,
$$\frac{K^N(\overleftarrow{X}^{\pi_{-\infty}}_N,y)}
{K^N(\overleftarrow{X}^{\pi_{-\infty}}_N,S)}\to\pi_{-\infty}(y)\mbox{ and }
\frac{K^N(\overleftarrow{X}^{\pi_{+\infty}}_N,y)}
{K^N(\overleftarrow{X}^{\pi_{+\infty}}_N,S)}\to\pi_{-\infty}(y)$$
where $\overleftarrow{X}^{\pi_{-\infty}}_N\to -\infty$
and $\overleftarrow{X}^{\pi_{+\infty}}_N\to +\infty$ almost surely.
Hence, from some $N$ onward,
$K^N(\overleftarrow{X}^{\pi_{+\infty}}_N,\cdot )/K^N(\overleftarrow{X}^{\pi_{+\infty}}_N,S)$
 dominates
$K^N(\overleftarrow{X}^{\pi_{-\infty}}_N,\cdot)
/K^N(\overleftarrow{X}^{\pi_{-\infty}}_N,S)$ stochastically.
We conclude under Condition [7] that $\pi_{+\infty}$ is stochastically bigger than $\pi_{-\infty}$.
Since any $\rho$-variant probability $\pi$ can be represented as a
 convex combination of the extremals:
$\pi=\alpha_{-}\pi_{-\infty}+\alpha_{+}\pi_{+\infty}$ by (73) in \cite{Dynkin}
we conclude all $\rho$-variant probabilities are stochastically bounded
above and below by $\pi_{+\infty}$ and $\pi_{-\infty}$ respectively.

To summarize
\begin{Pro}\label{extremes}
There are exactly two minimal points in the $\rho$-Martin entrance boundary corresponding to the measures
$\pi^0_{-\infty}$ and $\pi^0_{+\infty}$ corresponding to the geometric boundary  $\{-\infty, +\infty\}$.
If we renormalize these measures to probabilities we get $\pi_{-\infty}$ and $\pi_{+\infty}$. $\pi_{-\infty}$ and $\pi_{+\infty}$ are also extremal in the
sense that they are stochastically the smallest and largest $\rho$-invariant probabilities.
\end{Pro}

We can do a similar analysis of the $\rho$-Martin exit kernel $$k(x,y):=\frac{G_{x,y}(R)}{G_{x_0,y}(R)}.$$
As $y_n\to+\infty$, $k(x,y_n)\to h_{+\infty}(x)$ and as $y_n\to -\infty$, $k(x,y_n)\to h_{-\infty}(x)$
where $h_{+\infty}$ and $h_{-\infty}$ are the two extremals in the $\rho$-Martin exit boundary.
We conclude that we can identify the $\rho$-Martin exit boundary and the $\rho$-Martin entrance boundary with
the geometric boundary $\{-\infty, +\infty\}$. Hence Condition [4] holds.

 Even though $K$ is substochastic at $0$ we can still
 define a reversibility measure $\gamma$. For $x>0$ define
 $$\gamma(x)=\prod_{k=1}^x (p_{k-1}/q_k), \gamma(-x)=\prod_{k=-1}^{-x} (q_{k+1}/p_k),\gamma(0)=1.$$
 It is easy to check that $\gamma(x) K(x,y)=\gamma(y)K(y,x)$ for all $x,y$.

 It is easy to check that the mapping $h(x) \to \pi(x)=h(x)\gamma(x)$ provides a duality between $\rho$-harmonic functions $h$
 and $\rho$-invariant measures $\pi$. In particular this duality maps $h_{+\infty}$ to $\pi^0_{+\infty}$
 and $h_{-\infty}$ to $\pi^0_{-\infty}$. This follows because the time reversal with respect to
 the $\rho$-invariant measure
 $\gamma(x)h_{+\infty}$ is
 $$\frac{\gamma(y)h_{+\infty}(y)}{\rho \gamma(x)h_{+\infty}(x)}K(y,x)
 =\frac{\gamma(x)h_{+\infty}(y)}{\rho \gamma(x)h_{+\infty}(x)}K(x,y)
 =K(x,y)\frac{h_{+\infty}(y)}{\rho h_{+\infty}(x)}.$$
 This $h_{+\infty}$-transform of $K$ drifts to $+\infty$ so
 the time reversal does as well. This means $\gamma h_{+\infty}$ is an extremal associated with $+\infty$ and that means it is equal to $\pi^0_{+\infty}$.
 The same argument gives $\gamma h_{-\infty}=\pi^0_{-\infty}$.

 Let $A_{+}(0,y)$ denote the probability the $h_{+\infty}$-transform
 starting at $0$ reaches $y$ which is of course $1$. Let $A_{-}(0,y)$ denote the probability the $h_{-\infty}$-transform
 starting at $0$ reaches $y$ which tends to $0$ as $y\to\infty$. Hence
 $$1=A_{+}(0,y)=A_{-}(0,y)\frac{h_{+\infty}(y)}{h_{+\infty}(y)}$$
 so $\lim_{y\to +\infty}\frac{h_{-\infty}(y)}{h_{+\infty}(y)}=0.$
 Similarly $\lim_{y\to -\infty}\frac{h_{+\infty}(y)}{h_{-\infty}(y)}=0.$

To summarize
\begin{Pro}\label{extremes2}
There are exactly two minimal points in the $\rho$-Martin exit boundary corresponding to the $\rho$-harmonic functions
$h_{-\infty}$ and $h_{+\infty}$ corresponding to the geometric boundary  $\{-\infty, +\infty\}$. Moreover
$$\lim_{y\to +\infty}\frac{h_{-\infty}(y)}{h_{+\infty}(y)}=0\mbox{ and }
\lim_{y\to -\infty}\frac{h_{+\infty}(y)}{h_{-\infty}(y)}=0.$$
Also any $\rho$-harmonic function $h$ satisfies the relative Fatou theorem
relative to $h_{+\infty}$ and $h_{-\infty}$
\end{Pro}

\begin{proof}
We have established everything except the last statement.
By (61) in \cite{Dynkin} $h$ is a convex combinations of the extremals; i.e.
$h= \alpha_{-}h_{-\infty}+\alpha_{+}h_{+\infty}$. Hence
$$\lim_{y\to +\infty}\frac{h(y)}{h_{+\infty}(y)}
=\lim_{y\to +\infty}\frac{\alpha_{-}h_{-\infty}(y)+\alpha_{+}h_{+\infty}(y)}{h_{+\infty}(y)}
=\alpha_{+}.$$
Similarly
$$\lim_{y\to -\infty}\frac{h(y)}{h_{-\infty}(y)}
=\alpha_{-}.$$
\end{proof}

Condition [5] requires we check $K^n(x,S)/K^n(x_0,S)\to \hat{h}(x)/\hat{h}(x_0)$.
Lemma \ref{onekill} shows this is equivalent to checking
$K^n(x,x_0)/K^n(x_0,x_0)\to \hat{h}(x)/\hat{h}(x_0)$.
A consequence of Condition [5] and Lemma \ref{onekill} (recall $x_0=0$) is
\begin{eqnarray}
\frac{K^n(0,x)}{K^n(0,S)}&=&\frac{\gamma(x)}{\gamma(0)}\frac{K^n(x,0)}{K^n(0,S)}
\sim \frac{\gamma(x)}{\gamma(0)}\frac{K^n(x,S)}{K^n(0,S)}\frac{1-\rho}{\kappa}
\nonumber\\
&\sim & \frac{\gamma(x)}{\gamma(0)}\frac{\hat{h}(x)}{\hat{h}(0)}\frac{1-\rho}{\kappa}
\label{atzero}
\end{eqnarray}
This means the tight sequence of probabilities $K^n(0,x)/K^n(0,S)$ converges  to a probability proportional
to $\gamma(x)\hat{h}(x)$. Since a Yaglom limit starting from $0$ fails in Kesten's counterexample \cite{Kesten} this means the Jacka-Roberts Condition [5] also does not hold  in his example.

\subsection{Main result}
Recall the space-time kernel $\tilde{K}$ on $S\times (-\infty,0]$ defined in Subsection \ref{ST}.
The measure $P^N_x$ defined by $(\tilde{X}_0,\tilde{X}_1,\ldots,\tilde{X}_N)$ started at $(x,-N)$
  where $N=n+m$.  Since the  Jacka and Roberts Condition [5] holds, for any fixed $m$ as $N\to\infty$,
 $P^N_x$    converges to $P^{\hat{h}}_x$
 restricted to coordinates $0$ through $m$
where $P^{\hat{h}}_x$ is the measure derived from the $\hat{h}$-transformed Markov chain $X^{\hat{h}}_n; n=0,\ldots$ with kernel
$K_{\hat{h}}(x,y)=R K(x,y)\hat{h}(y)/\hat{h}(x)$. $X^{\hat{h}}$ is the chain conditioned to live forever.
Since $X$
is $R$-transient then $X^{\hat{h}}_n$ is transient.

By Theorem 4 in \cite{Dynkin} $X^{\hat{h}}_n\to X^{\hat{h}}_{\infty}$
(in the Martin exit topology) where $X^{\hat{h}}_{\infty}$ is a random variable with support on  $B=\{-\infty,+\infty\}$, the space of
exits in the $\rho$-Martin exit boundary.
 Let $\mu^{\hat{h}}_x$ be the distribution of $X^{\hat{h}}_{\infty}$
when $X^{\hat{h}}_0=x$.

\begin{Thm}\label{main}
If an irreducible, aperiodic $R$-transient chain on $\mathbb{Z}$  satisfies Conditions [1,2,3,5,6,7] with killing set $x_0=0$
 then
$$\frac{K^N(x,y)}{K^N(x,S)}\to \mu_x^{\hat{h}}(-\infty)\pi_{-\infty}(y)
+\mu_x^{\hat{h}}(+\infty)\pi_{+\infty}(y)=\pi_x(y).$$
\end{Thm}

This theorem generalizes the example in \cite{McFoley}.  Example \ref{twosided-asym}  also provide  chains where Conditions [1,2,3,5,6]
can be checked but in this case $\hat{h}=h_{+\infty}$ so $\mu_x^{\hat{h}}(+\infty)=1$ and
the above limit is $\pi_{+\infty}(y)$ which doesn't depend on $x$.
Before giving the proof of Theorem \ref{main} we first establish preliminary
lemmas.

\begin{Lem}\label{Ihope}
For any $N$, $\lim_{x\to+\infty}\sum_{n=0}^N \frac{R^n K^n(x,S)}{G_R(x,S)}=0$.
\end{Lem}
\begin{proof}
Using time reversal with respect to $\pi_{+\infty}$
\begin{eqnarray*}
\frac{R^n K^n(x,S)}{G_R(x,S)}&=&
\frac{\pi_{+\infty}\overleftarrow{K}^n(x)}{\pi_{+\infty}\overleftarrow{G}(x)}\\
&\to& 0 \mbox{ as }x\to\infty
\end{eqnarray*}
since $\pi_{+\infty}\overleftarrow{G}(x)$ tends to a constant because
$\overleftarrow{X}$ is transient to $+\infty$ and since clearly
$\pi_{+\infty}\overleftarrow{K}^n(x)\to 0$ as $x\to\infty$.
\end{proof}

\begin{Lem}\label{neglect}
If $\mu_x^{\hat{h}}(+\infty)>0$ then
$$\lim_{x\to +\infty}\lim_{n\to\infty}\lim_{m\to\infty}\sum_{z<x}\tilde{K}(x,-N;z,-m)=0.$$
If $\mu_x^{\hat{h}}(-\infty)>0$ then
$$\lim_{x\to -\infty}\lim_{n\to\infty}\lim_{m\to\infty}\sum_{z>x}\tilde{K}(x,-N;z,-m)=0.$$
\end{Lem}

\begin{proof}
For $n$ fixed $K^n(x,z)$ puts mass on a finite number of $z$'s
and by the Jacka-Roberts property
\begin{eqnarray*}
\lim_{m\to\infty}\sum_{z<x}\tilde{K}(x,-N;z,-m)&=&\lim_{m\to\infty}
\sum_{z<x}K^n(x,z)\frac{K^{m}(z,S)}{K^{m+n}(x,S)}\\
&=& \sum_{z<x}K^n(x,y)\frac{\hat{h}(z)}{\rho^n\hat{h}(x)}\\
&=&\sum_{z<x}K^n_{\hat{h}}(x,z).
\end{eqnarray*}
The $\hat{h}$-transformed chain is transient to $\{-\infty,+\infty\}$ so for any $\epsilon$ we can pick an $x$ sufficiently big so that
$P_x(X^{\hat{h}}_{\infty}=+\infty)\geq 1-\epsilon$. For this $x$,
$\lim_{n\to\infty}\sum_{z<x}K^n_{\hat{h}}(x,z)<\epsilon.$ As $x\to \infty$
we can pick $\epsilon$ arbitrarily small so the result follows for $x\to\infty$.
The proof for the case $x\to -\infty$ is analogous.
\end{proof}

Let $C_w=[w,+\infty)$
\begin{Lem}\label{hasone}
If $\mu_x^{\hat{h}}(+\infty)>0$ then uniformly in $w$,
$$\lim_{x\to +\infty}|\limsup_{N\to\infty}\frac{K^{N}(x,C_w)}{K^{N}(x,S)}
-\liminf_{N\to\infty}\frac{K^{N}(x,C_w)}{K^{N}(x,S)}|=0.$$
By the same proof picking $D_w=(-\infty,w]$, if $\mu_x^{\hat{h}}(-\infty)>0$ then uniformly in $w$
$$\lim_{x\to -\infty}|\limsup_{N\to\infty}\frac{K^{N}(x,D_w)}{K^{N}(x,S)}
-\liminf_{N\to\infty}\frac{K^{N}(x,D_w)}{K^{N}(x,S)}|=0.$$
\end{Lem}

\begin{proof}
Using Lemma \ref{neglect}, for any $\epsilon >0$ pick $L$, $A$ and $B$ sufficiently large such that
for $x>L$, $n>A$ and $m>B$
\begin{eqnarray}
\sum_{z<x}\tilde{K}(x,-N;z,-m)&=&
\sum_{z<x}K^n(x,z)\frac{K^{m}(z,S)}{K^{m+n}(x,S)}<\epsilon.\label{smallsum}
\end{eqnarray}

By Proposition \ref{produceit} any weakly convergent subsequence of
$K^N(x,\cdot)/K^N(x,S)$
converges to
a $\rho$-invariant probability. Suppose we can find subsequences
$I=\{m_i:i=1,\cdots\}$ and $J=\{N_j:1,\cdots\}$ such that
$K^{m_i}(x,\cdot)/K^{m_i}(x,S)$ converges to $\beta(\cdot)$ which is stochastically larger
than the limit $\alpha(\cdot)$ of $K^{N_j}(x,\cdot)/K^{N_j}(x,S)$.

Pick $m\in I$ and  sufficiently large so that  $m>B$ and  $K^{m}(x,C_w)/K^{m}(x,S)$
is within $\epsilon$ of $\beta(C_w)$.
Then pick a $N\in J$ sufficiently large so that $n>A$ where $N=n+m$ and
$K^{N}(x,C_w)/K^{N}(x,S)$ is within $\epsilon$ of $\alpha(C_w)$.

\begin{eqnarray}
\frac{K^{N}(x,C_w)}{K^{N}(x,S)}&=&\tilde{K}(x,-N;C_w,0)\\
&=&\sum_{z}\tilde{K}(x,-N;z,-m)\tilde{K}(z,-m;C_w,0)\label{split}
\end{eqnarray}
By (\ref{smallsum}),
$\sum_{z<x}\tilde{K}(x,-N;z,-m)\tilde{K}(z,-m;C_w,0)<\epsilon$.
Moreover for $z\geq x$ using the nearest neighbour property and aperiodicity
$\tilde{K}(z,-m:C_w,0)$ is  larger than
$\tilde{K}(x,-m:C_w,0)$.
This follows since the nearest neighbour non-homogeneous chain $\tilde{X}$
with kernel $\tilde{K}(x,-n;y,-n+1)=K(x,y)K^{n-1}(y,S)/K^n(x,S)$ on space-time
has probability $\tilde{s}_x=s_x K^{n-1}(x,S)/K^n(x,S)>s_x\geq 1/2$ of staying
put where we used Condition [7]. The monotonicity follows by the remark after
Condition [7] which applies even to non-homogeneous chains.

Hence
\begin{eqnarray*}
\frac{K^{N}(x,C_w)}{K^{N}(x,S)}&\geq& \tilde{K}(x,-m;C_w,0)-\epsilon
\geq \beta(C_w)-2\epsilon
\end{eqnarray*}
so $\alpha(C_w)\geq \beta(C_w)-3\epsilon$.  However $\beta$ is stochastically
larger than $\alpha$ so $|\beta(C_w)-\alpha(C_w)|<3\epsilon$.
Since $\epsilon$ can be arbitrarily small as $x\to +\infty$ and $m\to\infty$
we get our result.
\end{proof}

We now need to identify the Yaglom limits in Lemma \ref{hasone}.
\begin{Lem}\label{farout}
If $\mu_x^{\hat{h}}(+\infty)>0$ then
\begin{eqnarray*}
\lim_{x\to +\infty}\liminf_{N\to\infty} \frac{K^N(x,\cdot)}{K^N(x,S)}&=&\pi_{+\infty}(\cdot)
\end{eqnarray*}
If $\mu_x^{\hat{h}}(-\infty)>0$ then
\begin{eqnarray*}
\lim_{x\to -\infty}
\limsup_{N\to\infty}\frac{K^N(x,\cdot)}{K^N(x,S)}&=&\pi_{-\infty}(\cdot)
\end{eqnarray*}
\end{Lem}

\begin{proof}
For each $x$ let $\alpha_x$ denote $ \liminf_{N\to\infty} \frac{K^N(x,\cdot)}{K^N(x,S)}$ which is necessarily a $\rho$-invariant probability. The $\alpha_x$ are stochastically increasing in $x$. Suppose the limit as $x\to\infty$ is $\alpha$. By Lemma \ref{hasone} and the fact that
the measure $K^N(x,\cdot)/K^N(x,S)$ is stochastically increasing in $x$, for any $\epsilon$ and uniformly in $C_w=[w,\infty)$
we can
pick an $L$ and $K$ such that for all $x\geq L$ and all $N\geq K$,
$|\alpha(C_w)-\frac{K^N(x,C_w)}{K^N(x,S)}|\leq \epsilon$.
However $K^N(x,\cdot)/K^N(x,S)$ is stochastically smaller than $\alpha$
so this means
$||\alpha(\cdot)-\frac{K^N(x,\cdot)}{K^N(x,S)}||\leq
\sup_w|\alpha(C_w)-\frac{K^N(x,C_w)}{K^N(x,S)}|\leq\epsilon$
where $||\cdot||$ is the total variation.

For some $N>K$ we may
pick $x>L$ such that
$||\chi(x,\cdot)-\pi_{+\infty}(\cdot)||<\epsilon$ and
using Lemma \ref{Ihope}, such that
$\sum_{n=0}^N R^nK^n(x,S)/G_R(x,S)<\epsilon$.
Moreover,
\begin{eqnarray*}
||\chi(x,\cdot)-\alpha(\cdot)||&\leq&
\sum_{n=0}^N R^n\frac{K^n(x,S)}{G_R(x,S)}
+\sum_{y}|\alpha(y)-\sum_{n=N+1}^{\infty} R^n\frac{K^n(x,y)}{G_R(x,S)}|
\end{eqnarray*}
and since
$$\sum_{n=0}^{\infty} \frac{R^nK^n(x,S)}{G_R(x,S)}=1$$
\begin{eqnarray*}
\lefteqn{\sum_{y}|\alpha(y)-\sum_{n=N+1}^{\infty} R^n\frac{K^n(x,y)}{G_R(x,S)}|}\\
&\leq &\sum_{n=N+1}^{\infty} \frac{R^nK^n(x,S)}{G_R(x,S)}
\sum_{y}|\frac{K^n(x,y)}{K^n(x,S)}-\alpha(y)|+
\sum_{n=0}^N\frac{R^nK^n(x,S)}{G_R(x,S)}\\
&\leq& (1-\epsilon)\epsilon+\epsilon.
\end{eqnarray*}
Since $\epsilon$ is arbitrary $\alpha=\pi_{+\infty}$
and this gives the result.
\end{proof}

\begin{Cor}\label{hopeful}
If $\mu_x^{\hat{h}}(+\infty)>0$ then for any sequence $x_N\to +\infty$ such that
$$K^{N+1}(x_N,S)/K^N(x_N,S)\to \rho \mbox{ we have }
\lim_{N\to\infty} \frac{K^N(x_N,\cdot)}{K^N(x_N,S)}=\pi_{+\infty}(\cdot).$$
Similarly if $\mu_x^{\hat{h}}(-\infty)>0$ then for any sequence $y_N\to -\infty$ such that
$$K^{N+1}(y_N,S)/K^N(y_N,S)\to \rho\mbox{ we have }
\lim_{N\to\infty} \frac{K^N(y_N,\cdot)}{K^N(y_N,S)}=\pi_{-\infty}(\cdot).$$
\end{Cor}

\begin{proof}
By Lemma \ref{hasone} any weakly converging subsequence
of $K^N(x_N,\cdot)/K^N(x_N,S)$
converges to a $\rho$-invariant probability say $\alpha$. Pick an $x$ sufficiently large so that with
$\liminf_{N\to\infty} K^N(x,\cdot)/K^N(x,S)$
 is within $\epsilon$ of $\pi_{+\infty}$. Pick $K$
such that for $N\geq K$, $x_N\geq x$ and hence by the nearest neighbour property
$K^N(x_N,\cdot)/K^N(x_N,S)$ is stochastically bigger than
$K^N(x,\cdot)/K^N(x,S)$; i.e. $\alpha$ is stochastically larger than an $\epsilon$
perturbation of $\pi_{+\infty}$. Since $\epsilon$ is arbitrary and
$\pi_{+\infty}$ is the stochastically largest quasi-stationary distribution we
have our result. The proof for $y_N\to-\infty$ is the same.
\end{proof}

\begin{proof}[Proof of Theorem \ref{main}]
We just prove the case where both $\mu_x^{\hat{h}}(+\infty)>0$ and
$\mu_x^{\hat{h}}(-\infty)>0$.
 By Lemma \ref{farout} for any $\epsilon >0$ we can pick $v$ sufficiently large and $u$ sufficiently small
 along with $M$ sufficiently large such that for $m\geq M$
 that $\tilde{K}(y,-m;\cdot,0)= K^m(y,\cdot)/K^m(y,S)$ is within $\epsilon$ of $\pi_{+\infty}$ for all $y\geq v$ and is within $\epsilon$ of $\pi_{-\infty}$ for all $y\leq u$.

 Because $X^{\hat{h}}$ is $1$-transient we can pick an $n$ such that
 \begin{eqnarray*}
 |P_x(X^{\hat{h}}_{n}\geq v)-\mu_x^{\hat{h}}(+\infty)|&<&\epsilon\\
 |P_x(X^{\hat{h}}_{n}\leq u)=\mu_x^{\hat{h}}(-\infty)|&<&\epsilon.
 \end{eqnarray*}
 where $\mu_x^{\hat{h}}(-\infty)=P_x(X^{\hat{h}}_{\infty}=-\infty)$
 and $\mu_x^{\hat{h}}(+\infty)=P_x(X^{\hat{h}}_{\infty}=+\infty)$.

 Next pick $n\geq L$ sufficiently large
 the chain $\tilde{X}$  starting from $(x,-N)$ with $N=n+m$ agrees closely with the chain $X^{\hat{h}}$
 up to time $-m$.  i.e. such that
 $P_{(x,-N)}(X^{\hat{h}}_{k}\neq \tilde{X}_{k}, 0\leq k\leq n)|\leq \epsilon$.

 Consequently
 \begin{eqnarray*}
 \lefteqn{\frac{K^N(x,\cdot)}{K^N(x,S)}
 =P_x(\tilde{X}_N=\cdot)=\sum_z P_x(\tilde{X}_n=z)\tilde{K}(z,-m:\cdot,0)}\\
 &=&\sum_{z\geq v}P_x(X^{\hat{h}}_{n}=z)\cdot\pi_{+\infty}(\cdot)
 +\sum_{z\leq u}P_x(X^{\hat{h}}_{n}=z)\cdot\pi_{-\infty}(\cdot)
 +{\cal O}(\epsilon)\\
 &=&\mu_x^{\hat{h}}(+\infty)\cdot\pi_{+\infty}(\cdot)
 +\mu_x^{\hat{h}}(-\infty)\cdot\pi_{-\infty}(\cdot)
 +{\cal O}(\epsilon)
 \end{eqnarray*}
 Since $\epsilon$ is arbitrarily small we have our result.
 \end{proof}

Under the hypotheses of Theorem \ref{main} we have a good description
of the space-time Martin entrance boundary; i.e.
$$k((x,-N);(y,t))\to \rho^t\pi_x(y)\mbox{ as } N\to\infty$$
 where $$\pi_x(y)=\mu_x^{\hat{h}}(+\infty)\cdot\pi_{+\infty}(y)
 +\mu_x^{\hat{h}}(-\infty)\cdot\pi_{-\infty}(y).$$

 By Corollary \ref{hopeful} the extremal $\rho^t\pi_{+\infty}(y)$
 is associated with the entrance point which is the limit of $(x_N,-N)$ where
 $k((x_N,-N);(y,t))\to \pi_{+\infty}(y)$ with $x_N\to\infty$ and
 $K^{N+1}(x_N,S)/K^N(x_N,S)\to \rho$. Similarly the extremal $\rho^t\pi_{-\infty}(y)$ is associated with the entrance point  which is the limit of $(y_N,-N)$ where
 $$k((y_N,-N);(y,t))\to \rho^t\pi_{-\infty}(y)$$ with $y_N\to -\infty$ and
$K^{N+1}(y_N,S)/K^N(y_N,S)\to \rho$.

For an arbitrary $\rho$-invariant probability $\alpha$ apply Orey's Theorem to  the associated
time reversal $\overleftarrow{X}$ having kernel $\overleftarrow{K}$ with $\beta=\delta_y$.
It follows that a.s. $P_{\alpha}$,
$$\frac{K^N(\overleftarrow{X}_N,y)}{K^N(\overleftarrow{X}_N,S)}
\to \alpha(y)\frac{P_{y}(\overleftarrow{X}_{\infty}=b)}
 {P_{\alpha}(\overleftarrow{X}_{\infty}=b)}$$
 on the set where $\overleftarrow{X}\to b$.
 Hence,
 $$E_{\alpha}[\frac{K^N(\overleftarrow{X}_N,y)}{K^N(\overleftarrow{X}_N,S)}]
 \to\alpha(y)[P_{y}(\overleftarrow{X}_{\infty}=-\infty)
 +P_{y}(\overleftarrow{X}_{\infty}=+\infty)] =\alpha(y).$$
However by Orey's theorem we also have
$K^{N+1}(\overleftarrow{X}_N,S)/K^N(\overleftarrow{X}_N,S)\to \rho$
so
$$\frac{K^N(\overleftarrow{X}_N,y)}{K^N(\overleftarrow{X}_N,S)}\to
P_{\alpha}(\overleftarrow{X}_{\infty}=-\infty)\pi_{-\infty}(y)
+P_{\alpha}(\overleftarrow{X}_{\infty}=+\infty)\pi_{+\infty}(y).$$
We conclude
\begin{eqnarray}
\alpha(y)&=&P_{\alpha}(\overleftarrow{X}_{\infty}=-\infty)\pi_{-\infty}(y)
+P_{\alpha}(\overleftarrow{X}_{\infty}=+\infty)\pi_{+\infty}(y).\label{reduces}
\end{eqnarray}

We can obtain the above by the representation of the constant function
$h={\bf 1}$
which is harmonic for $\overleftarrow{K}$ as given at (55) in \cite{Dynkin}.
The Martin kernel associated with $\overleftarrow{K}$ is given in Theorem 11
in \cite{Dynkin} as $\chi(x,y)/\alpha(y)$ having standard measure
$\hat{\gamma}={\bf 1}\cdot \alpha$. In this case
$\mu^h$ in (55) is given by (37) in \cite{Dynkin} as
$$\mu^h(\Gamma)=P_{\hat{\gamma}}(\overleftarrow{X}_{\infty}\in \Gamma)
=P_{\alpha}(\overleftarrow{X}_{\infty}\in \Gamma)$$
Hence the representation
$${\bf 1}=\int_{\hat{B}}\frac{\chi(x,y)}{\alpha(y)}\mu^h(dy)$$
reduces to (\ref{reduces}).

\subsection{Verify Condition [5]}
We can use  Lemma \ref{onekill} as a simple method of verifying Condition [5].
\begin{Pro}\label{simpler}
For nearest neighbour chains on the integers with killing at a single point $x_0$ satisfying Condition [2]
it suffices to check
$$\frac{K^n(x_0+1,x_0)}{K^n(x_0,x_0)}\mbox{ or } \frac{K^n(x_0-1,x_0)}{K^n(x_0,x_0)}\mbox{ converges}$$
in order for Condition [5] to hold.
\end{Pro}

\begin{proof}
Suppose $K^n(x_0+1,x_0)/K^n(x_0,x_0) \to L$.
If Condition [1] holds then $K^{n+1}(x_0+1,x_0)/K^n(x_0+1,x_0)\to \rho$. But
\begin{eqnarray*}
\lefteqn{\frac{K^{n+1}(x_0+1,x_0)}{K^n(x_0,x_0)}=
K(x_0+1,x_0)\frac{K^n(x_0,x_0)}{K^n(x_0,x_0)}}\\
&+&K(x_0+1,x_0+1)\frac{K^n(x_0+1,x_0)}{K^n(x_0,x_0)}+K(x_0+1,x_0+2)\frac{K^n(x_0+2,x_0)}{K^n(x_0,x_0)}\\
&\to&K(x_0+1,x_0)+K(x_0+1,x_0+1)L\\
& &+K(x_0+1,x_0+2)
\lim_{n\to\infty}\frac{K^n(x_0+2,x_0)}{K^n(x_0,x_0)}.
\end{eqnarray*}
since the left hand side converges the limit $\lim_{n\to\infty}K^n(x_0+2,x_0)/K^n(x_0,x_0)$ must exist.
By iteration this shows $\lim_{n\to\infty}K^n(y,x_0)/K^n(x_0,x_0)$ exists for all $y$.
(\ref{atzero}) shows this limit is $ \hat{h}(y)/\hat{h}(x)$ and gives Condition [5].
\end{proof}

It may be the case that one point in the Martin boundary dominates.
In the two-sided Example \ref{twosided-asym} the following holds.
\begin{itemize}\label{oneside}
\item[Condition [8]]  Condition [5] holds and $\hat{h}=h_{+\infty}$.
\end{itemize}
This condition means that $+\infty$ is the dominant boundary point. In Theorem \ref{main} this means $\pi_x=\pi_{+\infty}$ so in these cases the Yaglom limit
does not depend on $x$.

\section{Examples}\label{stay}
\begin{Exp}\label{McDF}
We generalize the hub-and-spoke nearest neighbour chain on the integers with kernel $K$ in \cite{McFoley} by adding the probability $r$ of
staying put at every state. We get a modified kernel $K_r=rI+(1-r)K$
which satisfies the conditions of Theorem \ref{main} if $r\geq 1/2$.
\begin{Pro}
$\lim_{m\to\infty}\frac{K_r^m(x,y)}{K_r^m(x,S)}=\pi_x(y)$.
\end{Pro}
The proof below is valid for an arbitrary kernel $K$ of period $2$.
It could be extended to arbitrary period $d$.
\begin{proof}
Suppose $x$ is even. By \cite{McFoley}, for any $\epsilon$ there exists an $L$ such that
$m>L$ implies $$|\frac{K^{m}(x,y)}{K^{m}(x,S)}-\frac{\pi_x(y)}{\pi_x(2 \ZZ)}|\leq \epsilon$$ for $y$ even and $m$ even
and
$$|\frac{K^{m}(x,y)}{K^{m}(x,S)}-\frac{\pi_x(y)}{\pi_x(2 \ZZ + 1)}|\leq \epsilon$$
for $y$ odd and $m$ odd
and
$$\frac{K^{m+1}(x,S)}{K^{m}(x,S)}-\rho|\leq \epsilon.$$

Let $B(n,m)=\frac{m!}{n!(m-n)!}(1-r)^nr^{m-n}$.
Pick $M>L$  sufficiently large so that
$\sum_{n:|n-(1-r)m|>\delta m} B(n,m)<\epsilon$ for all $m\geq M$
where $0<\delta <1-r$ and $\delta< r$.

Note that for $n$ such that $|n-(1-r)m|\leq\delta m$,
$1-n/m\leq r+\delta$ so
$$\frac{m-n}{n+1}\leq \frac{1-n/m}{n/m}\leq \frac{r+\delta}{1-r-\delta}$$
and $1-n/m\geq r-\delta$ so
$$ \frac{m-n}{n+1}\geq\frac{1-n/m}{n/m+1/m}\geq \frac{r-\delta}{1-r +\delta+1/m}.$$
Picking $\delta$ small and $m>\tilde{M}$ where $\tilde{M}>M$ is sufficiently
large we have for $n$ satisfying $|n-(1-r)m|\leq\delta m$ that
$(m-n)/(n+1)=r/(1-r)+{\cal O}(\epsilon)$.

Suppose $y$ is even then for  $m>M$,
\begin{eqnarray*}
K_r^m(x,y)&=&\sum_{n=0}^m B(n,m)K^n(x,y)\\
&=&\sum_{n=0}^{N} B(n,m)K^n(x,y)
+\sum_{2k>N}^m B(2k,m)K^{2k}(x,y)\\
&=&\sum_{2k>N}^m B(2k,m)K^{2k}(x,y)+{\cal O}(\epsilon)\\
&=&\sum_{2k>N}^m B(2k,m) K^{2k}(x,S)\frac{\pi_x(y)}{\pi_x(2 \ZZ)}+{\cal O}(\epsilon)\\
&=&\frac{\pi_x(y)}{\pi_x(2 \ZZ)}\sum_{2k\geq 0}^m B(2k,m) K^{2k}(x,S)+{\cal O}(\epsilon).
\end{eqnarray*}

Moreover, letting $B=\sum_{2k\geq 0}^m B(2k,m)K^{2k}(x,S)$,
\begin{eqnarray*}
\lefteqn{K_r^m(x,S)=\sum_{2k\geq 0}^m B(2k,m)K^{2k}(x,S)+\sum_{2k+1\geq 0}^m B(2k+1,m)K^{2k+1}(x,S)}\\
&=&B+\sum_{k:|2k+1-(1-r)m|\leq \delta m} B(2k+1,m)K^{2k+1}(x,S)+{\cal O}(\epsilon)\\
&=&B+\sum_{k:|2k+1-(1-r)m|\leq \delta m}  B(2k+1,m)\rho K^{2k}(x,S)+{\cal O}(\epsilon)\\
&=&B+\sum_{k:|2k+1-(1-r)m|\leq \delta m}
 \rho \frac{(m-2k)}{2k+1}\frac{1-r}{r}B(2k,m) K^{2k}(x,S)+{\cal O}(\epsilon)\\
&=&B(1+\rho\frac{1-r}{r}\frac{r}{1-r})+{\cal O}(\epsilon)\mbox{ since }\frac{m-2k}{2k+1}=\frac{r}{1-r}+{\cal O}(\epsilon)\\
&=&(1+\rho)\sum_{2k\geq 0}^m (B(2k,m)K^{2k}(x,S)+{\cal O}(\epsilon)
\end{eqnarray*}

Using the fact that $\pi_x(2\ZZ)=1/(1+\rho)$ given in Theorem 1 in \cite{McFoley} and the fact the $\epsilon$ is arbitrarily small we conclude
\begin{eqnarray*}
\frac{K_r^m(x,y)}{K_r^m(x,S)}&=&\frac{1}{1+\rho}\frac{\pi_x(y)}{\pi_x(2 \ZZ)}=\pi_x(y).
\end{eqnarray*}

The cases when $y$ is odd or $x$ is odd follow in the same way. To
extend to a $d$ periodic case where $K$ causes transitions  through a sequence of subsets $S_0,S_1,\ldots S_{d-1},S_d=S_0$, assume $x$ is in class $S_0$ then $\pi_z(S_k)=c_k$ given by (17) in Theorem 1 in \cite{McFoley}.
\end{proof}
\end{Exp}

\begin{Exp}\label{twosided-asym}
Consider a nearest neighbour random walk on the integers with transitions
where, for $x>0$,
\begin{eqnarray*}
& K(x,x+1)=p,K(x,x-1)=q, \\
&K(-x,-x+1)=a,K(-x,-x-1)=b
\end{eqnarray*}
and  $K(0,1)=p$, $K(0,-1)=b$. We assume $p+q=1$, $p<q$ and $a+b=1$, $b<a$;
i.e. there is only killing at $0$ so Condition [6] holds. We also assume $\rho=2\sqrt{pq}>2\sqrt{ab}$
which implies $b<p<1/2$.

The $z$-transform of the recurrence time to $0$ for the $K$ kernel is
\begin{eqnarray*}
F_{0,0}(z)&=&zpF_{1,0}(R)+zb F_{-1,0}(R)\\
&=&zp\frac{(1-\sqrt{1-4pqz^2})}{2zp}+zb\frac{(1-\sqrt{1-4abz^2})}{2zb}
\end{eqnarray*}
as in \cite{Vere-Jones-Seneta}. Since $F_{0,0}(z)$ becomes singular at $z=R=(2\sqrt{pq})^{-1}$ and takes the
value
$V=1/2+(1-\sqrt{1-ab/pq})/2<1$ there we conclude the spectral radius of $K$ is $\rho_0=2\sqrt{pq}$
and $K$ is $R$-transient.

Starting from $0$, the chain my die in one step or can return to zero to try again.
Hence
$E_0 z^{\zeta}=(1-b-p)z+F_{0,0}(z)E_0 z^{\zeta}$ so
$$E_0 R^{\zeta}=\frac{(1-b-p)R}{1-V}<\infty.$$
Hence Condition [2] holds at $z=R$. Since $F_{1,0}(R)<\infty$ and $F_{-1,0}(R)<\infty$
Condition [2] holds for all $z$.

Condition [3] holds automatically for nearest neighbour random walks.

Consider $f(s)=bs^2-2\sqrt{pq} s+a=0$.
The roots are $t_0=\sqrt{pq}(1-\sqrt{1-ab/pq})/b$ and $t_1=\sqrt{pq}(1+\sqrt{1-ab/pq})/b$. Both roots are real since $ab<pq$. The mid point between the roots is $\sqrt{pq}/b>\sqrt{pq/ab}>1$. Since $f(0)=a$, $f'(0)<0$ and $f(1)=1-2\sqrt{pq}>0$ it follows that
both roots  are greater than one.
Notice that for either root $t$: $at^{-1}+bt=(a+bt^2)/t=2\sqrt{pq}=\rho$.
We want to find $\rho$-invariant measures for $K$ so
define
\begin{eqnarray*}
\mu(x) = \begin{cases}
(1+cx)\sqrt{\frac{p}{q}}^{x} \mbox{ where $c\geq 0$}&\mbox{if } x > 0 \\
d_0t_0^{x}+d_1 t_1^x \mbox{ where $d_0\geq 0$ and $d_0+d_1=1$} & \mbox{if } x <0\\
1& \mbox{if } x =0.
\end{cases}
\end{eqnarray*}
For $\mu$ to be positive we do not require $d_1\geq 0$.
To renormalize $\mu$ into a probability  $\pi=\mu/T(c)$ where $T(c)$ is the sum of all the mass:
$$T(c)=\frac{1}{1-\sqrt{p/q}}+c\sqrt{p/q}\frac{1}{(1-\sqrt{p/q})^2}+\frac{d_0 t_0^{-1}}{1-t_0^{-1}}+\frac{d_1 t_1^{-1}}{1-t_1^{-1}}.$$

Remark that for $x>0$
\begin{eqnarray*}
\lefteqn{\mu(x-1)p+\mu(x+1)q}\\
&=&(p\sqrt{\frac{p}{q}}^{x-1}+q\sqrt{\frac{p}{q}}^{x+1})+c(p(x-1)\sqrt{\frac{p}{q}}^{x-1}+q(x+1)\sqrt{\frac{p}{q}}^{x+1})\\
&=&2\sqrt{pq}(\sqrt{\frac{p}{q}}^{x}+cx\sqrt{\frac{p}{q}}^{x})\\
&=&\rho\mu(x).
\end{eqnarray*}
For $x<0$,
\begin{eqnarray*}
\mu(x-1)a+\mu(x+1)b&=&a(d_0t_0^{x-1}+d_1 t_1^{x-1} )+b(d_0t_0^{x+1}+d_1 t_1^{x+1}) \\
&=&d_0t_0^x(at_0^{-1}+bt_0)+d_1t_1^x(at_1^{-1}+bt_1)\\
&=&\rho\mu(x).
\end{eqnarray*}

In order to be $\rho$-invariant for  $x=0$ we need
\begin{eqnarray*}
2\sqrt{pq}\cdot 1&=&\mu(-1)a+\mu(1)q\\
&=&a(d_0t_0^{-1}+d_1 t_1^{-1} )+q(1+c)\sqrt{\frac{p}{q}} .
\end{eqnarray*}
Taking $d_1=1-d_0$ and solving $d_0$ in terms of $c$ we get
$$d_0=\frac{(1-c)\sqrt{pq}/a-t_1^{-1}}{t_0^{-1}-t_1^{-1}}.$$

The requirement $d_0\geq 0$ implies $c\leq 1-at_1^{-1}/\sqrt{pq}$.
The requirement $d_0\leq 1$ implies $c\geq 1-at_0^{-1}/\sqrt{pq}$.
Hence $1-at_0^{-1}/\sqrt{pq}\leq c\leq 1-at_1^{-1}/\sqrt{pq}$.

The upper bound is positive since
\begin{eqnarray}\label{pos}
1-\frac{at_1^{-1}}{\sqrt{pq}}&=&1-\frac{ab}{pq(1+\sqrt{1-ab/pq})}>0.
\end{eqnarray}
On the other hand the lower bound is negative since
\begin{eqnarray}\label{neg}
1-\frac{at_0^{-1}}{\sqrt{pq}}&=&1-\frac{ab}{pq}\frac{1}{1-\sqrt{1-ab/pq}}\\
&=&1-\frac{v}{1-\sqrt{1-v}}=1-\frac{v(1+\sqrt{1-v})}{1-(1-v)}=-\sqrt{1-v}
\end{eqnarray}
where $v=ab/pq<1$ so the above is negative.
However $c$ must be nonnegative or else $\mu$ is not a positive measure so
$0\leq c\leq 1-at_1^{-1}/\sqrt{pq}$.

One extremal of the family of possible $\mu$'s are given by $d_0=0$ with
$$c=c_1=1-\frac{at_1^{-1}}{\sqrt{pq}}
=1-\frac{ab}{pq(1+\sqrt{1-ab/pq})}=\sqrt{1-ab/pq}$$
 corresponding to the measure
  \begin{eqnarray*}
\mu_{+\infty}(x) = \begin{cases}
(1+c_1x)\sqrt{\frac{p}{q}}^{x} &\mbox{if } x > 0 \\
 t_1^x  & \mbox{if } x <0\\
1& \mbox{if } x =0.
\end{cases}
\end{eqnarray*}
The other extremal is given by $c=c_0=0$ and
$d_0=\frac{\sqrt{pq}-at_1^{-1}}{a(t_0^{-1}-t_1^{-1})}$. The later is positive because
\begin{eqnarray*}
\sqrt{pq}-at_1^{-1}&=&\sqrt{pq}-\frac{ab}{\sqrt{pq}(1+\sqrt{1-ab/pq})}\\
&=&\sqrt{pq}(1-\frac{v}{(1+\sqrt{1-v})}=\sqrt{pq}\sqrt{1-v}>0
\end{eqnarray*}
where $v=ab/pq$. The corresponding measure is
 \begin{eqnarray*}
\mu_{-\infty}(x) = \begin{cases}
\sqrt{\frac{p}{q}}^{x} &\mbox{if } x > 0 \\
d_0t_0^{x}+(1-d_0) t_1^x  & \mbox{if } x <0\\
1& \mbox{if } x =0.
\end{cases}
\end{eqnarray*}

 The time reversal with respect to $\mu_{-\infty}$ where $d_0>0$ is given   by
\begin{eqnarray*}
\overleftarrow{K}(x,x+1)&=&\frac{\mu_{-\infty}(x+1)}{\rho\mu_{-\infty}(x)}K(x+1,x)\mbox{ for } x\leq -1\\
&\sim & \frac{t_0b}{2\sqrt{pq}} \mbox{ for large $|x|$}
\end{eqnarray*}
Also,
\begin{eqnarray*}
\overleftarrow{K}(x,x-1)&=&\frac{\mu_{-\infty}(x-1)}{\rho\mu_{-\infty}(x)}K(x-1,x)\mbox{ for } x\leq -1\\
&\sim & \frac{t_0^{-1}a}{2\sqrt{pq}} \mbox{ for large $|x|$}
\end{eqnarray*}
Note that
\begin{eqnarray*}
\frac{t_0b}{2\sqrt{pq}}-\frac{t_0^{-1}a}{2\sqrt{pq}}&=&\frac{1}{2t_0\sqrt{pq}}(bt_0^2-a)\\
&=&\frac{1}{2t_0\sqrt{pq}}(2t_0\sqrt{pq}-2a)=1-\frac{at_0^{-1}}{\sqrt{pq}}<0\mbox{ by (\ref{neg}).}
\end{eqnarray*}
We conclude the time reversal of $K$  drifts to $-\infty$ from $x$ sufficiently negative.

For $x\geq 1$,
\begin{eqnarray*}
\overleftarrow{K}(x,x+1)&=&\frac{\mu_{-\infty}(x+1)}{\rho\mu_{-\infty}(x)}K(x+1,x)=\frac{1}{2}
\end{eqnarray*}
and
\begin{eqnarray*}
\overleftarrow{K}(x,x-1)&=&\frac{\mu(x-1)}{\rho\mu(x)}K(x-1,-x)
= \frac{1}{2}.
\end{eqnarray*}
This means that the above time reversed kernel is null recurrent on the positive integers so the time reversal  of $K$ with respect to $\mu_{-\infty}$ is transient and drifts to $-\infty$

On the other hand the time reversal with respect to $\mu_{+\infty}$ when $b_0=0$ is given  (for $x\leq -1$) by
\begin{eqnarray*}
\overleftarrow{K}(x,x+1)&=&\frac{\mu_{+\infty}(x+1)}{\rho\mu_{+\infty}(x)}K(x+1,x)\mbox{ for } x\leq -1\\
&\sim & \frac{t_1b}{2\sqrt{pq}}
\end{eqnarray*}
Also,
\begin{eqnarray*}
\overleftarrow{K}(x,x-1)&=&\frac{\mu_{+\infty}(x-1)}{\rho\mu_{+\infty}(x)}K(x-1,x)\mbox{ for } x\leq -1\\
&\sim & \frac{t_1^{-1}a}{2\sqrt{pq}}
\end{eqnarray*}
Note that
\begin{eqnarray*}
\frac{t_1b}{2\sqrt{pq}}-\frac{t_1^{-1}a}{2\sqrt{pq}}&=&\frac{1}{2t_1\sqrt{pq}}(bt_1^2-a)\\
&=&\frac{1}{2t_1\sqrt{pq}}(2t_1\sqrt{pq}-2a)=1-\frac{at_1^{-1}}{\sqrt{pq}}>0\mbox{ by (\ref{pos}).}
\end{eqnarray*}

Moreover,
\begin{eqnarray}
\overleftarrow{K}(x,x+1)&=&\frac{\mu_{+\infty}(x+1)}{\rho\mu_{+\infty}(x)}K(x+1,x)\mbox{ for } x\geq 1\nonumber\\
&\sim & \frac{1}{2}\frac{(1+c_1(x+1))}{1+c_1x}\label{totheright}.
\end{eqnarray}
Also,
\begin{eqnarray}
\overleftarrow{K}(x,x-1)&=&\frac{\mu_{+\infty}(x-1)}{\rho\mu_{+\infty}(x)}K(x-1,x)\mbox{ for } x\geq 1\nonumber\\
&\sim & \frac{1}{2}\frac{(1+c_1(x-1))}{1+c_1x}\label{totheleft}.
\end{eqnarray}
A random walk with this kernel reflected at $0$ is transient to $+\infty$ and since the reversed chain
drifts toward $0$ on the negative integers we conclude the time reversal is transient to $+\infty$.

The kernel $K$ has period 2 and it is convenient to look at the even chain on the even integers; i.e. every two steps
 as we did in \cite{McFoley}.
Condition [1] holds since $K^2(0,0)=ba+pq>0$, $K^2(x,x)=2ba$ for $x<0$
and $K^2(x,x)=2pq$ for $x>0$.
Since the even chain is nearest neighbour (on the even integers) Condition  [4] holds automatically.

Finally as remarked in Proposition \ref{simpler}
since absorption only occurs in state $0$ we can check Condition [5]
if
$K^{2n}(x,0)/K^{2n}(0,0)$ converges to a limit if $x$ is even.
We prove this limit exists by brute force.

Note that $K^{2n}(0,0)$ is the coefficient of $z^{2n}$ in the generating function $G_{0,0}(z)$ and
\begin{eqnarray*}
G_{0,0}(z)&=&\frac{1}{1-F_{0,0}(z)}\\
&=&\frac{\sqrt{1-4abz^2}-\sqrt{1-4pqz^2}}{2(pq-ab)z^2}.
\end{eqnarray*}
So,  asymptotically in $n$, $K^{2n}(0,0)$  is the coefficient of order $z^{2n+2}$ in $-\sqrt{1-4pqz^2}$
divided by $2(pq-ab)$.
Next recall
$$\sqrt{1+x}=\sum_{n=0}^{\infty}\left( \begin{array}{c} 1/2\\n\end{array}\right)x^n$$
where $$\left(\begin{array}{c}1/2\\n\end{array}\right)=(-1)^{n-1}\frac{(2n-3)!}{2^{2n-2}n!(n-2)!}\sim
-(-1)^n\frac{1}{2\sqrt{\pi} n^{3/2}}$$
by Stiring's formula.
Substituting $x=-4pqz^2$, we have that the coefficient of $z^{2n+2}$ in $-\sqrt{1-4pqz^2}$ is asymptotically
$$(-1)(-1)(-1)^{n+1}(-1)^{n+1}\frac{1}{2\sqrt{\pi} (n+1)^{3/2}}(4pq)^{n+1}\sim
 \frac{2pq}{\sqrt{\pi}}\frac{(4pq)^n}{n^{3/2}}.$$
We conclude
$$K^{2n}(0,0)\sim \frac{pq}{ pq-ab}\frac{1}{\sqrt{\pi}}\frac{(4pq)^n}{n^{3/2}}$$
which is the same as (35) in \cite{Vere-Jones-Seneta}  at $i=j=1$.

On the other hand, for $x$ even, $K^{2n}(-x,0)$ is the coefficient $z^{2n}$ in the generating function
\begin{eqnarray}
F_{-x,0}(z)G_{0,0}(z)&=&\left(\frac{1-\sqrt{1-4abz^2}}{2zb}\right)^x
\frac{\sqrt{1-4abz^2}-\sqrt{1-4pqz^2}}{2(pq-ab)z^2}\label{goodone}
\end{eqnarray}
First consider the coefficient of $z^{2n+2+x}$ in
\begin{eqnarray}
\lefteqn{-\left(\frac{1-\sqrt{1-4abz^2}}{2b}\right)^x\cdot\frac{\sqrt{1-4pqz^2}}{2(pq-ab)}}\label{firstone}\\
&=&-\left(\frac{1-\sqrt{1-\frac{ab}{pq}w}}{2b}\right)^x\cdot\frac{1}{2(pq-ab)}(1-w)^{1/2}\label{thisone}\\
&=&C(1-w)^{1/2}+{\cal O}|1-w|^{1}\nonumber
\end{eqnarray}
$$\mbox{where }w=4pqz^2 \mbox{ and }C=-\left(\frac{1-\sqrt{1-\frac{ab}{pq}}}{2b}\right)^x\frac{1}{2(pq-ab)}.$$
This is true since $pq>ab$ so
\begin{eqnarray*}
\lefteqn{\frac{1-\sqrt{1-4\frac{ab}{pq}w}}{2b}\cdot\frac{1}{2(pq-ab)}-C}\\
&=&{\cal O}|1-w|^{1}
\end{eqnarray*}
in a neighbourhood of $w=1$.

Using Darboux's method based on the Riemann-Lebesgue Lemma as in \cite{woess} 16.8 it follows that
 the term in $w^n$ in (\ref{thisone}) is  asymptotically the same as  the term in $w^n$ in
 $C(1-w)^{1/2}$; i.e.
 $$-C(-1)^n\frac{1}{2\sqrt{\pi} n^{3/2}}(-1)^nw^n=
 -C\frac{1}{2\sqrt{\pi} n^{3/2}}(4pq)^n z^{2n}.$$
 It follows that the coefficient of $z^{2n+2+x}$ in (\ref{firstone}) and hence of $z^{2n}$ in the dominant term in
(\ref{goodone}) is
 $$-C\frac{1}{2\sqrt{\pi} (n+1+x/2)^{3/2}}(4pq)^{n+1+x/2}.$$

The coefficient of $z^{2n+2+x}$ in
$$\left(\frac{1-\sqrt{1-4abz^2}}{2bz}\right)^x\cdot\frac{\sqrt{1-4abz^2}}{2(pq-ab)z^2}$$ is clearly given by
 (35) in \cite{Vere-Jones-Seneta}  at $i=2,j=1$ and hence is of lower order $(4ab)^n/n^{3/2}$.

 We conclude that for $x$ even
 \begin{eqnarray*}
 \frac{K^{2n}(-x,0)}{K^{2n}(0,0)}
 &\sim& -C\frac{1}{2\sqrt{\pi} n^{3/2}}(4pq)^{n+1+x/2}
 \frac{ pq-ab}{pq}\sqrt{\pi}(4pq)^{-n}n^{3/2}\\
 &=&(4pq)^{x/2}\left(\frac{1-\sqrt{1-\frac{ab}{pq}}}{2b}\right)^x=t_0^x.
 \end{eqnarray*}

 Hence for $x<0$, $\hat{h}(x)=t_0^{-x}$ and we
note that $\hat{h}(x)$ for $x\leq 0 $ is $\rho$-harmonic and can be extended $\rho$-harmonically
 to $\hat{h}(x)=(1+cx)\sqrt{q/p}^x$ for $x\geq 0 $ provided
 $\rho \hat{h}(0)=b \hat{h}(-1)+p\hat{h}(1)$; i.e. $\rho=bt_0+(1+c)\sqrt{pq}$ or
 $$c=\frac{\rho-b t_0}{\sqrt{pq}}-1=2-(1-\sqrt{1-\frac{ab}{pq}})-1=\sqrt{1-\frac{ab}{pq}}>0.$$

$\hat{h}$ is precisely $h_{+\infty}$ which we can calculate from
$h_{+\infty}(x)=\mu_{+\infty}(x)/\gamma(x)$ where the reversibility measure $\gamma(x)=(q/p)^x$ for $x\geq 0$ and
$\gamma(x)=(b/a)^{|x|}$ for $x<0$. For $x<0$ this gives
$h_{+\infty}(x)=t_1^x (b/a)^x=t_0^{-x}=\hat{h}(x)$.
We conclude Condition [7] holds.

If we restrict to even integers and even times we can check Conditions [1-7] so a Yaglom limit holds.
The extension to the odds follows from the periodic analysis in \cite{McFoley}.
\end{Exp}

\begin{Exp}[Kesten's example]
Kesten \cite{Kesten} considers a chain on the integers much like the example
in \cite{McFoley} but with probability $r_x$ of staying put at each $x$.
Kesten picks the $r_x$ so that the Yaglom limit $K^n(0,\cdot)/K^n(0,S)$ fails.
As already noted at (\ref{atzero})
 the Jacka-Roberts property must fail as well.

The cone of $\rho$-invariant probabilities cannot be empty by Proposition \ref{extremes} and there exist two extremals $\pi_{+\infty}$ and $\pi_{-\infty}$ associated with points $+\infty$ and $\-\infty$ in the $\rho$-Martin entrance
boundary.

Now consider the associated space-time process. By Proposition \ref{produceit}
we can characterize the space-time entrance boundary as limits of sequences
$(x_i,-n_i)$ converging in the Martin topology satisfying
(\ref{oreysequence}) as points associated with product measures; i.e. of the form
$\rho^{t}\alpha(y)$ where $\alpha$ is a $\rho$-invariant probability.
By uniform aperiodicity the extremal invariant measures are
of product form: $\rho^{t}\pi_{+\infty}(y)$ and $\rho^{t}\pi_{-\infty}(y)$.  The associated Orey paths $\overleftarrow{X}^{\pi_{+\infty}}$ and
$\overleftarrow{X}^{\pi_{-\infty}}$ converge respectively to the points in the space-time Martin entrance boundary associated with these two extremals. It is not known if an analogue of
Corollary \ref{hopeful} is true.

The Jacka-Roberts Condition fails but in some sense Kesten constructs a sequence
of Jacka-Roberts limits. Kesten defines
 the $r_x$ to be constant on certain
intervals. More specifically, he chooses integers
$ a_1 = 1 < a_2 < a_3\cdots$ , $b_1 = 1 < b_2 <\cdots$
and numbers $c_k, d_k \in [1/4 ,1/2 ]$ and then takes
$$ r_x = c_k \mbox{ for }a_k\leq x < a_{k+1}\mbox{ and }
r_x=d_k \mbox{ for }-b_{k+l}<x \leq -b_k, k \geq 1.$$
The $c_k$ and $d_k$ are chosen to satisfy
$ r_0<  d_1 < c_1 < d_2 < c_2 < \cdots $
and $a_k \leq b_k,  k \geq 1$. Over the  interval $a_k\leq x <b_k\leq a_{k+1}$, $r_x=c_k$
while for $ -b_{k}<x \leq -a_k\leq -b_{k-1}$, $r_x=d_{k-1}<c_k$.

The interval $(-b_k,b_k)$
was chosen to dwarf the interval $(-a_k,a_k)$ so for transitions inside
$(-b_k,b_k)$ the waiting time $r_x=c_k$ in the interval $a_k\leq x<b_k$ dominates and the chain behaves like
a chain with $r_x=c_k$ for $a_k\leq x <\infty$ and
$r_x=d_{k-1}$ for $-\infty<x\leq -b_{k-1}$. This two-sided chain will behave
like Example \ref{twosided-asym}.  It will satisfy a Jacka-Roberts condition
with a limit we can denote by $\hat{h}_k$ which will be equal to the extremal
$\hat{h}_k^{+\infty}$; i.e. the $\hat{h}_k$-transform will diverge to $+\infty$.
Hence the Yaglom ratios of this two-sided chain will tend to the extremal $\pi^k_{+\infty}$ where $\pi^k_{+\infty}\to \pi_{+\infty}$ as $k\to\infty$.

Now consider the next even more gigantic interval $(-a_{k+1},a_{k+1})$
where $r_x=d_k$ on the interval $-b_{k+1}\leq -a_{k+1}<x\leq -b_k$ which dominates over the interval $(-a_{k+1},a_{k+1})$.
As above for transitions inside $(-a_{k+1},a_{k+1})$ the Yaglom ratios
will be close to $\pi^k_{-\infty}$ where $\pi^k_{-\infty}\to \pi_{-\infty}$ as $k\to\infty$. Kesten constructs the ever increasing
intervals where the Yaglom ratios swap back and and forth
between being close to $\pi_{+\infty}$ and $\pi_{-\infty}$.
This way the Yaglom limit fails.

\end{Exp}

\section*{Acknowledgements}
DMcD thanks Servet Martinez  for his hospitality  at the Center for Mathematical Modeling at the University of Chile
and Jaime San Martin for his mathematical help.

\bibliographystyle{apt}

\noindent   Robert D. Foley\\
   Industrial and Systems Engineering \\
   Georgia Institute of Technology\\
   Atlanta, GA 30332-0205 \\
   USA
  rfoley@gatech.edu
  
\vspace{.5cm}

\noindent David R. McDonald\\
   Department of Mathematics and Statistics\\
   The University of Ottawa\\
   Ottawa, Ontario\\
   Canada K1N6N5
 dmdsg@uottawa.ca
%

\end{document}